\documentclass{amsart}

\usepackage{amsmath}
\usepackage{amssymb}
\usepackage{amsthm}
\usepackage{mathrsfs}
\usepackage{mathtools}
\usepackage{tikz}
\usepackage{cleveref}
\usepackage{tikz-cd}
\usepackage{mathtools}

\usepackage{placeins}

\usepackage{float}

\usepackage{comment}

\theoremstyle{definition}
\newtheorem{thm}{Theorem}[section]
\crefname{thm}{Theorem}{Theorems}

\newtheorem{prop}[thm]{Proposition}
\crefname{prop}{Proposition}{Propositions}
\newtheorem{lem}[thm]{Lemma}
\crefname{lem}{Lemma}{Lemmas}
\newtheorem{clm}[thm]{Claim}
\newtheorem{conj}[thm]{Conjecture}
\newtheorem{quest}[thm]{Question}
\newtheorem{defn}[thm]{Definition}

\crefname{defn}{Definition}{Definitions}

\crefname{con}{Construction}{Constructions}
\newtheorem{exmp}[thm]{Example}

\newtheorem{rmk}[thm]{Remark}

\newtheorem*{ack*}{Acknowledgements}

\newtheorem{manualtheoreminner}{Theorem}
\newenvironment{manualtheorem}[1]{%
  \renewcommand\themanualtheoreminner{#1}%
  \manualtheoreminner
}{\endmanualtheoreminner}

\newcommand{\CC}{C}

\newcommand*{\mb}[1]{\mathbb{#1}}

\newcommand{\mc}{\mathcal}

\newcommand*{\on}[1]{\operatorname{#1}}

\newcommand{\Sym}{{\operatorname{Sym}}}
\newcommand{\Spec}{{\operatorname{Spec}}}
\newcommand{\1}{\textbf{1}}

\newcommand{\rood}[1]{\textcolor{red}{[#1]}}

\begin{document}

\author{Hunter Spink, Dennis Tseng}
\title{Incidence strata of affine varieties with complex multiplicities}
\thanks{%Affiliation:\,Harvard University. 
\texttt{Hunter Spink: hspink@math.harvard.edu, Dennis Tseng: DennisCTseng@gmail.com}}
\begin{abstract}
    To each affine variety $X$ and $m_1,\ldots,m_k\in \mathbb{C}$ such that no subset of the $m_i$ add to zero, we construct a variety which for $m_1,\ldots,m_k \in \mathbb{N}$ specializes to the closed $(m_1,\ldots,m_k)$-incidence stratum of $\Sym^{m_1+\ldots+m_k}X$. These fit into a finite-type family, which is functorial in $X$, and which is topologically a family of $\mathbb{C}$-weighted configuration spaces. We verify our construction agrees with an analogous construction in the Deligne category $\on{Rep}(S_{d})$ for $d \in \mathbb{C}$.
    
    We next classify the singularity locus and branching behaviour of colored incidence strata for arbitrary smooth curves. As an application, we negatively answer a question of Farb and Wolfson concerning the existence of an isomorphism between two natural moduli spaces.
   
\end{abstract}
\maketitle

\section{Introduction}
\noindent The moduli space of $d$ unordered points on $\mathbb{A}^1$ over $\mathbb{C}$ is $$\Sym^d\mathbb{A}^1\cong \mathbb{A}^d=\{z^d+a_1z^{d-1}+\ldots+a_d\mid a_1,\ldots, a_d \in \mathbb{C}\},$$
where the correspondence associates to a degree $d$ monic polynomial $f(z)$ its unordered set of roots. For a partition $P=(m_1,\ldots,m_k)$ of $d$, there is an associated \emph{incidence stratum}\footnote{We always mean \emph{closed} incidence stratum, and will always omit the word closed.} $$\Delta^k_P(\mathbb{A}^1):= \{(z-x_1)^{m_1}\ldots (z-x_k)^{m_k}\mid x_1,\ldots,x_k \in \mathbb{C}\}\subset \Sym^d \mathbb{A}^1,$$
which corresponds to the subvariety of $d$-point configurations on $\mathbb{A}^1$ whose incidences are at least as coarse as $P$. The study of the geometry and defining equations of the incidence strata $\Delta^k_P(\mathbb{A}^1)$ dates back to Cayley \cite{Cayley}, and has been extensively studied since \cite{Weyman1,Weyman2,Weyman3,Chipalkatti1,Abdesselam1,Abdesselam2,Kurmann,LS16}.

\begin{exmp}
\label{cubicex}
Let $P=(2,1)$. Then $\Delta^2_{(2,1)}(\mathbb{A}^1)$ is the cubic discriminant hypersurface
$$a_1^2a_2^2+18a_1a_2a_3-4a_2^3-4a_1^3a_3-27a_3^2=0.$$
The normalization is the bijective morphism
\begin{align*}
\Phi_{(2,1)}:\mathbb{A}^1\times \mathbb{A}^1 &\to \Delta^3_{(2,1)}(\mathbb{A}^1)\\
(a,b)&\mapsto (z-a)^2(z-b),
\end{align*}
and despite being a homeomorphism in the Euclidean topology, $\Phi_{(2,1)}$ is not an isomorphism as $\mathbb{A}^1\times \mathbb{A}^1$ is non-singular but $$\textrm{Sing}(\Delta^2_{(2,1)})=\{(z-a)^3 \mid a \in \mathbb{C}\}=\Delta^1_{(3)},$$ occuring when the doubled and single root collide.
%An application of Zariski's main theorem characterizes the singular locus of $\Delta_P(\mathbb{A}^1)$ as when the differential of the normalization map drops rank. In particular, the singular locus of $\Delta_P(\mathbb{A}^1)$ will itself be a union of incidence strata. 
This reflects a key difference between the topology and the algebraic geometry of incidence strata.
\end{exmp}

Even though the incidence strata $\Delta^k_{P}(\mathbb{A}^1)$ are defined with $m_1,\ldots,m_k$ positive integers, one can actually make sense of the case where $m_1,\ldots,m_k$ are complex numbers with no subset summing to zero \cite{Kasatani,Sergeev,Etingof}, and these generalized incidence strata for fixed $k$ fit into a finite-type family $\Delta^k(\mathbb{A}^1)$ over the space of allowable $m_i$ \cite[Proposition 2.6]{Sam}. The goal of this paper is to expand upon this theory in a number of different directions with concrete applications to algebraic geometry in mind.

First, we show for fixed $k$ that the incidence strata $\Delta^k_P(X)$ of arbitrary affine varieties $X$ fit into a finite-type family $\Delta^k(X)$ exactly as with $\mathbb{A}^1$ (\Cref{mainthm}), and this construction is functorial in $X$. The construction from \cite[Proposition 2.6]{Sam} yields for fixed $k$ a family of embeddings $\Delta^k_P(\mathbb{A}^1)\subset\mathbb{A}^{N_k}$ where $N_k$ depends only on the number of parts $k$ of the partition $P$. We will see a similar uniform embedding result holds for $\Delta^k_P(X)$ (\Cref{smallembedprop}). The proof that $N_k$ exists is not constructive \cite[Remark 2.7 (1)]{Sam}; we conjecture $N_k = 2^k-1$ (\Cref{HuntersConjecture}), verify $N_k=2^k-1$ for $k\leq 3$, and show $N_4 \ge 15$.

We next extend the method of determining the singular locus of $\Delta^k_P(\mathbb{A}^1)$ \cite{Chipalkatti1,Kurmann} to find the singular locus of colored incident strata in arbitrary smooth curves (\Cref{singthm}). Using this we answer a question of Farb and Wolfson \cite[Question 1.4]{Resultant} about whether certain moduli spaces are isomorphic (\Cref{notiso}).

Finally, Pavel Etingof communicated to us an analogous construction using the Deligne category $\on{Rep}(S_d)$ for $d \in \mathbb{C}$. The Deligne-categoric framework allows us to naturally consider $d$'th powers of varieties for $d\in \mathbb{C}$, and the construction mirrors the classical construction of $\Delta^k_{(m_1,\ldots,m_k)}(X)$ as the $S_{m_1+\ldots+m_k}$-quotient of the $(m_1,\ldots,m_k)$-incidence loci of ordered configurations in $X^{m_1+\ldots+m_k}$. In \Cref{Dcategories}, we verify that our interpolated varieties $\Delta^k_{(m_1,\ldots,m_k)}(X)$ arise as the spectrum of the algebra produced by this construction.

\begin{rmk}
    As we saw in \Cref{cubicex}, there are invariants of incidence strata which can be detected algebraically but not topologically. The topology of an incidence strata is determined by the lattice of partial sums of the $m_i$, yielding finitely many homeomorphism types, but as we will see if $k=3$ and $X=\mathbb{A}^n$ there are infinitely many isomorphism types of $k$-part incidence strata (\Cref{notisotrivprop}). Surprisingly when $k=2$ there are only two isomorphism types of incidence strata occurring in $\mathbb{A}^n$, classified by whether or not $m_1$ and $m_2$ are equal (\Cref{isotrivprop}), but we believe that this is particular to $\mathbb{A}^n$.
\end{rmk}
%Surprisingly we find there are only two isomorphism classes of $2$-part incidence strata in $\mathbb{A}^n$ depending on whether the parts are distinct or equal (reflected as we will see in the isotriviality of our constructed family), but for $k\geq 3$ such a phenomenon does not occur. Thus for $k\ge 3$ we see a stark contrast between the finite number of homeomorphism types and the infinite number of isomorphism types of $k$-part incidence strata.

\subsection{Statement of Results}
\label{statementresults}
We briefly sketch how one can generalize the incidence stratum $\Delta^k_P(\mathbb{A}^1)$ to the case where the weights $m_1,\ldots,m_k$ are complex numbers. If $m_1,\ldots,m_k \in \mathbb{N}$, then we may express

$$(z-x_1)^{m_1}\ldots(z-x_k)^{m_k}=z^{d}-a_1z^{d-1}+\ldots$$
where
\begin{align*}a_i=e_i(\underbrace{x_1,\ldots,x_1}_{m_1},\ldots,\underbrace{x_k,\ldots,x_k}_{m_k})=\sum_{i_1+\ldots+i_k=d}\binom{m_1}{i_1}\ldots\binom{m_k}{i_k}x_1^{i_1}\ldots x_k^{i_k}.\end{align*}

Note that for $m_i \in \mathbb{N}$, the expressions $a_i$ vanish for $i>d$.

For complex weights, these expressions for the $a_i$ still make sense, although the $a_i$ are in general non-zero for all $i$. One can show \cite{Kasatani,Sergeev} that if we restrict the $m_i$ to lie in the subset
$$(\mathbb{C}^k)^{\circ}=\mathbb{C}^k\setminus \bigcup_{S \subset \{1,\ldots,k\}}\{\sum_{i \in S}m_i=0\},$$ then despite not stabilizing to zero for sufficiently large $i$ the $a_i$ eventually depend polynomially on the previous $a_j$, so truncating at this point loses no information. Furthermore, there is a uniform bound for when this occurs \cite{Sam}, allowing one to use the truncated list of coefficients to define a family containing all $k$-part incidence strata on $\mathbb{A}^1$.

Formally, as long as no subset of the $m_i$ sum to zero, the subalgebra $\mathbb{C}[a_1,a_2,\ldots]\subset \mathbb{C}[x_1,\ldots,x_k]$ is finite type. In fact, letting the $m_i$ vary, we have a finite type subalgebra
\begin{align*}
    \mathbb{C}[m_1,\ldots,m_k][a_1,\ldots][\{\frac{1}{\sum_{i\in S}m_i}\}_{S}]\subset \mathbb{C}[m_1,\ldots,m_k,x_1,\ldots,x_k][\{\frac{1}{\sum_{i\in S}m_i}\}_{S}]
\end{align*}
whose inclusion is finite \cite[Proposition 2.6]{Sam} (here $S$ ranges over nonempty subsets of $\{1,\ldots,k\}$). Taking $\Spec$ shows these generalized incidence strata all fit together into a family, embedded in the same affine space $\mathbb{A}^{N_k}$ where $N_k$ is the point at which $m_1,\ldots,m_k,a_1,\ldots,a_{N_k}$ generate the entire sub-algebra.

We show that these considerations can be generalized to arbitrary affine varieties. 
\subsubsection{Incidence strata of affine varieties with complex multiplicities}

\begin{thm}
\label{mainthm}
There is a functorial assignment $X \mapsto \Delta^k(X)$ of affine varieties to affine varieties over $(\mathbb{C}^k)^{\circ}$ such that the reduced fiber over any $P=(m_1,\ldots,m_k) \in \mathbb{N}^k\subset \mathbb{C}^k$ is precisely the $(m_1,\ldots,m_k)$-incidence strata $$\Delta^k_P(X)\subset\Sym^{m_1+\ldots+m_k}X.$$
\end{thm}

%Given any finite-type $\mathbb{C}$-algebra $A$, our construction takes $\mathbb{C}[m_1,\ldots,m_k][\{\frac{1}{\sum_{i \in S}m_i}\}_{S}]\otimes A^{\otimes k}$ generated by $m_1,\ldots,m_k$ and all expressions of the form $\sum_{i=1}^k m_i(1\otimes \ldots \otimes a \otimes \ldots \otimes 1)$ where $a$ appears in the $i$'th coordinate.

\Cref{mainthm} generalizes almost verbatim to \emph{colored incidence strata}, where we will be especially interested in the case of smooth curves $\CC$. Consider again for simplicity the case $\CC=\mathbb{A}^1$, and suppose we have colors $1,\ldots,r$ and $d_i$ points of each color $i$. Then given $\overline{P}=(\overline{m}_1,\ldots,\overline{m}_k)\in (\mathbb{Z}_{\ge 0}^r\setminus \{\overline{0}\})^k$ such that $$\overline{m}_1+\ldots+\overline{m}_k=(d_1,\ldots,d_r),$$ we have an analogously defined \emph{$r$-colored incidence strata}
\begin{align*}
    &\Delta^{k,r}_{\overline{P}}\subset \prod_{i=1}^r \Sym^{d_i}\mathbb{A}^1 \cong \prod_{i=1}^r \mathbb{A}^{d_i}\\
    &\Delta^{k,r}_{\overline{P}}:=\{(\prod_{i=1}^k(z-x_i)^{(\overline{m_i})_1},\ldots,\prod_{i=1}^k(z-x_i)^{(\overline{m_i})_r})\mid x_1,\ldots,x_k \in \mathbb{C}\}
\end{align*}
where the $\overline{m}_i$ now not only keeps track of how many colored points are incident but also how many of each color. We show that the colored incidence strata $\Delta^{k,r}_{\overline{P}}(\mathbb{A}^1)$ can be similarly placed into a natural finite-type family
\begin{center}
\begin{tikzcd}
\Delta^{k,r}(\mathbb{A}^1)\ar[d]&\\
((\mathbb{C}^r)^k)^{\circ}\ar[equal]{r}&(\mathbb{C}^r)^k\setminus \bigcup_{S \subset \{1,\ldots,k\}} \sum_{i \in S} \{\overline{m}_i=0\},
\end{tikzcd}
\end{center}
so we may talk about $\Delta^{k,r}_{\overline{P}}(\mathbb{A}^1)$ for arbitrary $\overline{P}\in ((\mathbb{C}^r)^k)^{\circ}$. The following theorem specializes to \Cref{mainthm} when $r=1$ (i.e. the ``uncolored case'').

\begin{manualtheorem}{\ref{mainthm}'}\label{maincolthm}
There is a functorial assignment $X \mapsto \Delta^{k,r}(X)$ of affine varieties to varieties over $((\mathbb{C}^r)^k)^{\circ}$ such that the reduced fiber over any $\overline{P}=(\overline{m_1},\ldots,\overline{m_k})\in (\mathbb{Z}_{\ge 0}^r\setminus \{\overline{0}\})^k \subset ((\mathbb{C}^r)^k)^{\circ}$ is precisely the $r$-colored $(\overline{m_1},\ldots,\overline{m_k})$-incidence strata
$$\Delta^{k,r}_{\overline{P}}(X)\subset \Sym^{(\overline{m_1})_1+\ldots+(\overline{m_k})_1}X\times \cdots \times \Sym^{(\overline{m_1})_r+\ldots+(\overline{m_k})_r}X.$$
\end{manualtheorem}

\noindent\textit{A topological model.}
One can construct the underlying topological spaces of $\Delta^k_{P}(X)$ and $\Delta^{k,r}_{\overline{P}}(X)$ in the Euclidean topology as the quotient of the space of configurations of $k$ ordered points on $X$ (which we think of as carrying multiplicities associated to $P$ or $\overline{P}$ respectively), by the relation that two configurations $\mathcal{C}_1$ and $\mathcal{C}_2$ are equivalent if at each point of $x\in X$, the sum of the weights associated to the two configurations are the same. In other words, when points collide their multiplicities add. This process also works relatively over the space of allowable weights, and creates a natural family of weighted configuration spaces over this base. In \Cref{topologicalmodel}, we prove that our constructions agree with these topological notions.

\subsubsection{Singularities of incidence strata of curves with complex multiplicities}
Our next result determines the singularity and branching locus of the fiber $\Delta^{k,r}_{\overline{P}}(\mathbb{A}^1)$ (and more generally $\Delta^{k,r}_{\overline{P}}(\CC)$) with $\overline{P} \in ((\mathbb{C}^r)^k)^{\circ}$. \Cref{singthm} vastly generalizes the characterizations in \cite{Chipalkatti1,Kurmann} of singularity loci and branching behaviour of uncolored incidence strata of $\mathbb{A}^1$ (see \Cref{intsingrmk}).

\begin{thm}
\label{singthm}
Let $\CC$ be a smooth curve, let $\overline{P}=(\overline{m_1},\ldots,\overline{m_k})\in ((\mathbb{C}^r)^k)^{\circ}$, and let $\overline{X}\in \Delta^{k,r}_{\overline{P}}(\CC)$ be a point with associated multiplicities $\overline{Q}=(\overline{q_1},\ldots,\overline{q_\ell})\in ((\mathbb{C}^r)^\ell)^{\circ}$. 
\begin{itemize}
\item Preimages $\overline{Y}$ of $\overline{X}$ in the normalization of $\Delta^{k,r}_{\overline{P}}(\CC)$ (which correspond to branches at $\overline{X}$) are in bijection with systems of equations
\begin{align*}
    \overline{q_1}&=\overline{p_{1,1}}+\ldots+\overline{p_{1,i_1}}\\
    \overline{q_2}&=\overline{p_{2,1}}+\ldots+\overline{p_{2,i_2}}\\
    &\vdots\\
    \overline{q_\ell}&=\overline{p_{\ell,1}}+\ldots+\overline{p_{\ell,i_\ell}}
\end{align*}
where we consider for each $j$ the collection of $\overline{p_{j,1}},\ldots,\overline{p_{j,i_j}}$ only up to permutation, such that $i_1+\ldots+i_\ell=k$, and the collection of all $\overline{p_{i,j}}$ constitute the collection $\overline{m_1}\ldots,\overline{m_k}$ with multiplicities.
\item The normalization of $\Delta_{\overline{P}}^{k,r}(C)$ is smooth, and the normalization map is not an immersion near a preimage $\overline{Y}$ of $\overline{X}$ (i.e. the corresponding branch is not smooth at $\overline{X}$) if and only if there exists a $j$ such that the set of $\overline{p_{j,u}}$ with $1\le u \le i_j$ \emph{excluding repetitions} are linearly dependent.
\end{itemize}
\end{thm}

\begin{exmp}
Consider the following examples:
\begin{enumerate}
    \item If $\overline{m}_1=\begin{pmatrix} 1 \end{pmatrix}$, $\overline{m}_2=\begin{pmatrix} 2 \end{pmatrix}$, $\overline{q}_1=\begin{pmatrix} 3 \end{pmatrix}$, then there is one preimage of $q_1$ in the normalization, and the normalization map is not a immersion at that point as the two vectors $\begin{pmatrix} 1 \end{pmatrix}$ and $\begin{pmatrix} 2 \end{pmatrix}$ are not linearly independent. This is \Cref{cubicex} above when $\CC=\mathbb{A}^1$.
    \item If $\overline{m_1}=\begin{pmatrix} 1 \end{pmatrix}$, $\overline{m}_2=\begin{pmatrix} 1 \end{pmatrix}$, $\overline{q}_1=\begin{pmatrix} 2 \end{pmatrix}$, then there is one preimage of $\overline{q_1}$ in the normalization, and the normalization map is an immersion at that point as the single vector $\begin{pmatrix} 1 \end{pmatrix}$ is linearly independent. In this case, the normalization is simply the identity map on the symmetric square of $\CC$. 

    \item If $\overline{m_1},\ldots,\overline{m_5}$ are the columns of $\begin{pmatrix} 1 & 1 & 2 & 0 & 100\\ 1 & 1 & 0 & 2 & 101\end{pmatrix}$ and $\overline{q_1},\overline{q_2}$ are the columns of $\begin{pmatrix} 102 & 2\\ 103 & 2\end{pmatrix}$, then there are two preimages in the normalization corresponding to 
    \begin{alignat*}{2}
        \overline{q_1} &= \overline{m_1}+\overline{m_2}+\overline{m_5}   &&\qquad\overline{q_1} = \overline{m_3}+\overline{m_4}+\overline{m_5}\\
         \overline{q_2} &= \overline{m_3}+\overline{m_4}   &&\qquad\overline{q_2} = \overline{m_1}+\overline{m_2}.      
    \end{alignat*}
    The expressions for $\overline{q_1}$ and $\overline{q_2}$ on the left correspond to a preimage where the normalization map is an immersion and the expressions for $\overline{q_1}$ and $\overline{q_2}$ on the right correspond to a preimage where the normalization map is not an immersion.
\end{enumerate}
\end{exmp}

\begin{rmk}\label{intsingrmk}
Since the normalization of $\Delta^{k,r}_{\overline{P}}(\CC)$ is smooth (\Cref{singsect}), its smooth locus coincides with its normal locus by Zariski's main theorem. Thus, a point $\overline{X}\in \Delta^{k,r}_{\overline{P}}(\CC)$ is non-singular if and only if it has exactly one preimage in the normalization and the normalization map is an immersion near that preimage. \Cref{singthm} thus gives a combinatorial description of the singular locus of $\Delta^{k,r}_{\overline{P}}(\CC)$.
\end{rmk}

\noindent\textbf{A Conjecture of Farb and Wolfson.}
As an application, we negatively answer a question of Farb and Wolfson \cite{Resultant} concerning the existence of an isomorphism between two natural moduli spaces. In \cite{Resultant}, it was asked whether the varieties
\begin{align*}
Rat^*_{d,n} &= \{\text{Degree $d$ pointed regular maps $\mathbb{P}^1 \to \mathbb{P}^n$} \}\\
Poly^{d(n+1),1}_{n+1} &= \{\text{Monic degree $d(n+1)$ polynomials with no root of multiplicity $n+1$}\}
\end{align*}
are isomorphic when $n \ge 2$.

We prove the following.

\begin{thm}
\label{notiso}
The varieties $Rat^*_{d,n}(\mathbb{C})$ and $Poly^{d(n+1),1}_{n+1}(\mathbb{C})$ are not isomorphic for any $d,n \ge 2$.
\end{thm}

Note that Remark 1.5.1 in \cite{Resultant} shows that for $n=1$ and $d \ge 2$ the varieties have different fundamental groups. For $d=1$ the varieties are isomorphic by an identical argument to the one given in \cite[Remark 1.5.2]{Resultant}. Thus we in fact have a complete classification of when these varieties are isomorphic.

The proof also shows that they in fact can't even be biholomorphic to each other, but the following remains open.

\begin{quest}
Are $Rat^*_{d,n}(\mathbb{C})$ and $Poly^{d(n+1),1}_{n+1}(\mathbb{C})$ homeomorphic over $\mathbb{C}$?
\end{quest}

\subsubsection{Connection to Deligne categories}
Deligne categories $\on{Rep}(S_d)$ \cite{D07} for $d \in \mathbb{C}$ are rigid tensor categories which have a notion of ``complex tensor power of an algebra'', allowing us to consider objects like $\on{Spec}A^{\otimes d}$ parameterizing $d$ ordered points on $\on{Spec}A$ for $d\in \mathbb{C}$. A construction for the interpolated varieties $\Delta^k_{(m_1,\ldots,m_k)}$ via $\on{Rep}(S_d)$ was communicated to us by Pavel Etingof, which we will verify agrees with our construction. The ideal $I(m_1,\ldots,m_k)\subset A^{\otimes d}$ used in the following theorem will be defined and motivated in \Cref{Con3}.

\begin{thm}
\label{comparison}
Let $A$ be a reduced finite-type $\mathbb{C}$-algebra, and let $(m_1,\ldots,m_k)\in (\mathbb{C}^k)^{\circ}$ be such that $m_1,\ldots,m_k$ and $d=m_1+\ldots+m_k$ are all not in $\mathbb{Z}_{\ge 0}$. Then $$\Spec(\on{Hom}_{\on{Rep}(S_d)}(\1,A^{\otimes d}/I(m_1,\ldots,m_k)))\cong \Delta^k_{(m_1,\ldots,m_k)}(\Spec A)$$
\end{thm}

The condition on $d$ and the $m_i$ not being in $\mathbb{Z}_{\ge 0}$ does not present an obstacle, since one can simply scale $d$ and the $m_1,\ldots,m_k$ by the same number. 
\begin{rmk}
When $d\in \mathbb{Z}_{\ge 0}$, $\on{Rep}(S_d)$ is no longer semisimple and one needs to take a further quotient to get the usual category of representations of $S_d$ \cite[Section 3.4]{CO11}. We expect this does not present a serious issue for the direct construction, see \cite[Proposition 5.1]{D07} and \cite[Remark 4.1.2]{EA14}. 
\end{rmk}

%In fact, we show that the algebra of functions on both sides in \Cref{comparison} are equal to the subalgebra of $A^{\otimes k}$ generated by elements of the form
%$$\sum_{i=1}^km_i(1^{\otimes i-1} \otimes a \otimes 1^{\otimes k-i})$$
%with $a \in A$, and the Deligne construction yields this subalgebra even when $A$ is not reduced.

\subsubsection{Effective Bounds for Incidence Strata}
The construction from \cite[Proposition 2.6]{Sam} yields for fixed $k$ a family of embeddings $\Delta^k_P(\mathbb{A}^1)\subset\mathbb{A}^{N_k}$ (see the beginning of \Cref{statementresults} for a more precise definition of $N_k$). Here, we give results and conjectures towards computing $N_k$. Define $$N_k(P)=\min\{n \mid\Delta^k_P\text{ embeds into }\mathbb{A}^n\}.$$
Then we have
$$\max \{N_k(P)\mid P \in \mathbb{N}^k\}\le \max \{N_k(P)\mid P \in (\mathbb{C}^k)^{\circ}\}\le N_k.$$
\begin{conj}
\label{HuntersConjecture}
These three quantities equal $2^{k}-1$ for $k\geq 1$.
\end{conj}

\begin{thm}
\label{Conjthm}
\Cref{HuntersConjecture} is true for $k \le 3$, and $N_4((1,2,4,8))=15$.
\end{thm}
We believe that the maximum of $N_k(P)$ is attained for $P=(2^0,2^1,\ldots,2^{k-1})$.
As we will see (\Cref{smallembedprop}), $N_k(P)$ not only bounds the complexity of $k$-part incidence strata $\Delta^k_P(\mathbb{A}^1)$, but also the complexity of $k$-part incidence strata $$\Delta^k_P(X)\subset \Sym^{m_1+\ldots+m_k} X$$
for affine $X$. Indeed, if $X \subset \mathbb{A}^n$, then we may embed $$\Delta^k_P(X) \subset \mathbb{A}^{\binom{N_k(P)+n}{n}-1}.$$
Using $N_k(P)\le N_k$ this yields for fixed $k$ a uniform bound which drastically improves the dimension given by the Chow embedding $$\Delta^k_P(X)\subset \Sym^{m_1+\ldots+m_k}X\subset \mathbb{A}^{\binom{m_1+\ldots+m_k+n}{n}-1}$$
which depends on the total multiplicity $d=\sum_{i=1}^k m_i$.

We are able to produce an upper bound of approximately $k!^k$ on $N_k(P)$ and $2^{k^2}k!^k$ on $N_k$ using the effective Nullstellensatz \cite[Theorem 1.3]{Jelonek}, but with a slightly worse bound of approximately $k!^{2k}$, we can explicitly eliminate coefficients in the Chow embedding to produce our embedding. These yield recursive identities similar in spirit to the Newton identities, but for symmetric polynomials whose variables are specialized to $k$ groups of distinct values with $m_i$ variables in each group (\Cref{explicitsect}), which may be of independent interest.

\subsection{Acknowledgements}
We would like to thank Pavel Etingof, Benson Farb, Alexander Smith, and Marius Tiba for helpful conversations.

\section{Three Constructions for Interpolating $k$-part incidence strata}\label{3Con}

In this section, we outline three ways to interpolate $k$-part incidence strata of an affine variety $X$, which we will later show in \Cref{Dcategories} yield the same result. For simplicity we describe here all of the constructions in the uncolored case, and in future sections only detail the extension of this first construction to the colored case. We end this section with a discussion of the subtleties of reducedness for the fibers of $\Delta^k(X)$ over $(\mathbb{C}^k)^{\circ}$. We could extend all of our constructions to $\Delta^k(\Spec A)$ for $A$ non-reduced, but we do not understand the scheme-theoretic fibers over $P\in (\mathbb{C}^k)^{\circ}$ well enough to refine \Cref{mainthm} in this direction.

\subsection{Construction 1: Deforming the Chow embedding}

This first construction is the most straightforward, and the most ad hoc construction. Our plan will be to first construct $\Delta^k(\mathbb{A}^n)$ by deforming the Chow embedding, and then define $\Delta^k(X)$ as the closed subvariety where the points are restricted to lie on $X$.

Recall that for an $n$-dimensional vector space $V$, the Chow embedding takes
\begin{align*}\Sym^d \mathbb{A}(V) &\hookrightarrow \prod_{i=1}^d \Sym^i V\\
(v_1,\ldots,v_d)^{S_d}&\mapsto \text{Coefficients of }(z-v_1)\ldots (z-v_d),\end{align*}
in exact analogue to the identification of $\Sym^d\mathbb{A}^1$ with the set of degree $d$ monic polynomials in an indeterminate $z$ (see \Cref{gelflem}). Therefore the $(m_1,\ldots,m_k)$-incidence strata corresponds to
$$\Delta^k_{(m_1,\ldots,m_k)}(\mathbb{A}(V))=\{(z-v_1)^{m_1}\ldots (z-v_k)^{m_k}\mid v_1,\ldots,v_k \in V\}\subset \prod_{i=1}^d\Sym^iV$$
where $d=m_1+\ldots+m_k$.
Exactly as we sketched in the introduction in the case of $\mathbb{A}^1$ from \cite{Sam}, $\Delta^k_{(m_1,\ldots,m_k)}(\mathbb{A}(V))$ projects isomorphically onto its image in $\prod_{i=1}^{N_k}\Sym^i V$, and now we have a bounded list of coefficients which we can interpolate to complex $m_i$. Writing $a_i$ for the coefficient of $z^{d-i}$ in $(z-v_1)^{m_1}\ldots (z-v_k)^{m_k}$, which depends polynomially on $m_1,\ldots,m_k,v_1,\ldots,v_k$, we will show the morphism
\begin{align*}
    (\mathbb{C}^k)^{\circ}\times \mathbb{A}(V)^k &\to (\mathbb{C}^k)^{\circ}\times \prod_{i=1}^{N_k}\Sym^i V\\
    (m_1,\ldots,m_k,v_1,\ldots,v_k) &\mapsto (m_1,\ldots,m_k,a_1,\ldots,a_{N_k})
\end{align*}
is finite and birational, so in particular the image is closed. In particular, the reduced fiber of the image over any $(m_1,\ldots,m_k)\in (\mathbb{C}^k)^{\circ}$ is precisely $\Delta^k_{(m_1,\ldots,m_k)}(\mathbb{A}(V))$ projected to $\prod_{i=1}^{N_k}\Sym^iV$, which as we said before is isomorphic to $\Delta^k_{(m_1,\ldots,m_k)}(\mathbb{A}(V))$. Therefore we may take $\Delta^k(\mathbb{A}(V))$ to be the image.

Finally, given an affine variety $X$, we choose an embedding $X \hookrightarrow \mathbb{A}(V)$, and take $\Delta^k(X)$ to be the image of the proper composite map
$$(\mathbb{C}^k)^{\circ}\times X^k \to (\mathbb{C}^k)^{\circ}\times \mathbb{A}(V)^k \to (\mathbb{C}^k)^{\circ}\times \prod_{i=1}^{N_k}\Sym^iV.$$

We will later show that this is functorial and works as expected topologically.

\subsection{Construction 2: Interpolating the symmetric quotient}

The next construction is manifestly functorial, and relies on the identification of $\Delta^k_{(m_1,\ldots,m_k)}(X)$ as the $S_d$-quotient of the corresponding ordered $(m_1,\ldots,m_k)$-incidence strata in $X^d$, where $d=m_1+\ldots+m_k$. Here, the ordered $(m_1\ldots,m_k)$-incidence strata in $X^d$ is the union of all embedded copies of $X^k$ formed by partitioning $\{1,\ldots,d\}=M_1\sqcup \ldots \sqcup M_k$ with $|M_i|=m_i$, and taking
$$\underbrace{X\times \ldots \times X}_k \to X^{M_1}\times \ldots \times X^{M_k}=X^d$$
where the first map is the product of diagonal embeddings. Writing $X=\Spec A$, we now exploit the fact that the subalgebra $(A^{\otimes d})^{S_d}\subset A^{\otimes d}$ is the subalgebra generated by elements of the form $\sum_{i=1}^d 1^{\otimes i-1}\otimes a \otimes 1^{\otimes d-i}$. The ideal of each embedded copy of $X^k$ has the same intersection with $(A^{\otimes d})^{S_d}$, so the algebra of functions on $\Delta^k_{(m_1,\ldots,m_k)}(X)$ is the cokernel of the composite
$$(A^{\otimes d})^{S_d}\hookrightarrow A^{\otimes d}=A^{\otimes m_1}\otimes \ldots \otimes A^{\otimes m_k} \to \underbrace{A \otimes \ldots \otimes A}_k=A^{\otimes k}$$
where the final map is the tensor product of the multiplication maps $A^{\otimes m_i}\to A$. Thus, we see that  $\Delta^k_{(m_1,\ldots,m_k)}(X)$ is the spectrum of the subalgebra of $A^{\otimes k}$ generated by expressions of the form
$$\sum_{i=1}^k m_i (1^{\otimes i-1}\otimes a \otimes 1^{\otimes k-i})$$
with $a \in A$. Finally, we take $\Delta^k(X)$ to be the spectrum of the subalgebra of $\Gamma((\mathbb{C}^k)^{\circ})\otimes A^{\otimes k}$ generated by $m_1,\ldots,m_k$ and all expressions of the above form.
\begin{rmk}
It is not clear from this construction that the resulting algebra is finite-type, nor that taking the reduced fiber over $(m_1,\ldots,m_k)\in (\mathbb{C}^k)^{\circ}$ yields $\Delta^k_{(m_1,\ldots,m_k)}(X)$. However, it is easy to show that this construction agrees with Construction 1 (see \Cref{Dcategories}), for which these properties are checked.
\end{rmk}

\subsection{Construction 3: Deligne categorical interpolation}
\label{Con3}

The final construction, communicated to us by Pavel Etingof, uses the natural interpolation provided by the Deligne category $\on{Rep}(S_d)$ to construct the algebra of functions on $\Delta^k_{(m_1,\ldots,m_k)}(X)$ with $(m_1,\ldots,m_k)\in (\mathbb{C}^k)^{\circ}$. It would be interesting to check the entire interpolated variety $\Delta^k(X)$ can be obtained through Deligne categories, but the verification and construction would be necessarily more complicated.

Given a commutative algebra $A$ and $d \in \mathbb{C}$, we can construct the algebra $A^{\otimes d}$ in the Ind-completion of the Deligne category $\on{Rep}(S_d)$ \cite[Section 4.1]{E14}. There is a natural multiplication map of algebras $A^{\otimes d}\to A$ in $\on{Ind}(\on{Rep}(S_d))$, whose kernel is an ideal $J_d$. Intuitively, we should think of $A^{\otimes d}/J_d$ as the algebra of functions on the diagonal of ``$(\Spec A)^d$'' where ``all $d$ points are equal''. 

Given $m_1,\ldots,m_k\in \mathbb{C}$ with $m_1+\ldots+m_k=d$, there is a natural restriction functor \cite[Section 2.3]{E14} $\on{Res}:\on{Rep}(S_d)\to \on{Rep}(S_{m_1})\boxtimes \ldots \boxtimes \on{Rep}(S_{m_k})$ taking
$$A^{\otimes d} \mapsto A^{\otimes m_1}\boxtimes \ldots \boxtimes A^{\otimes m_k}.$$  
\begin{defn}
Let $I(m_1,\ldots,m_k)$ be the sum of all ideals $J\subset A^{\otimes d}$ such that $$\on{Res}(J)\subset J_{m_1}\boxtimes \ldots \boxtimes J_{m_k}.
$$
\end{defn}
The quotient $A^{\otimes d}/I(m_1,\ldots,m_k)$ should be thought of as the algebra of functions on the configurations of points in ``$(\Spec A)^d$'' where the $d$ points that are grouped with multiplicities $m_1,\ldots,m_k$. Letting $\1$ be the object in $Rep(S_d)$ corresponding to the trivial representation, $\on{Hom}(\1, -)$ should be thought of as taking $S_d$-invariants, so our third construction will be the spectrum of
$$\on{Hom}_{\on{Rep}(S_d)}(\1,A^{\otimes d}/I(m_1,\ldots,m_k)),$$ 
which is an honest $\mathbb{C}$-algebra which should morally be the algebra of functions on $\Delta^k_{(m_1,\ldots,m_k)}(X)$.
In fact, when $m_1,\ldots,m_k,d \not \in \mathbb{Z}_{\ge 0}$, we will show in \Cref{Dcategories} that Constructions 2 and 3 agree for arbitrary (possibly nonreduced and non-finite-type) $\mathbb{C}$-algebras $A$.

\subsection{On the fibers of $\Delta^k(X)$}
\label{fibersubsection}
%In \Cref{fibersubsection}, we will let $A$ be a reduced, finite type, $k$-algebra. If $A$ is nonreduced, we can still make sense of the statements in this section, but we choose to add the reduced hypothesis for convenience. 

\Cref{mainthm} shows the reduced fibers of the family $\Delta^k(X)\to (\mathbb{C}^k)^{\circ}$ agree with $\Delta^k_P(X)$ for $P\in (\mathbb{C}^k)^{\circ}$. However, if we take scheme-theoretic fibers, some of the fibers will be nonreduced.
\begin{exmp}
\label{2clumpexample}
We will see in \Cref{isotrivprop} that $$\Delta^2(\mathbb{A}^1)\cong\mathbb{C}[m_1,m_2,\frac{1}{m_1m_1(m_1+m_2)}][x,y]/((m_1-m_2)^2x^3=y^2).$$ Taking the fiber over $(m_1,m_2)=(a,a)$ yields the doubled line $\mathbb{C}[x,y]/(y^2)$. Concretely, this non-reducedness occurs because there is a relation between the coefficients of $(z-x_1)^a(z-x_2)^a=z^{2a}+b_1z^{2a-1}+\ldots$ which does not occur through specializing a relation between the coefficients of $(z-x_1)^{m_1}(z-x_2)^{m_2}$, namely $$\binom{1/a}{1}b_3+\binom{1/a}{2}(2b_1b_2)+\binom{1/a}{3}b_1^3=0.$$
One can check that the square of this equation does arise from such a specialization.
\end{exmp}
In spite of this, it is easy to see by the fact that geometric-reducedness of fibers is an open condition on the base \cite[EGA IV33, 12.2.4]{EGAIV3} and from Noetherian induction that there exists a partition of $(\mathbb{C}^k)^{\circ}$ into locally closed sets so that $\Delta^k(X)$ has reduced fibers over $(\mathbb{C}^k)^{\circ}$ when restricted to the preimage of any of the parts. In particular, replacing $\Delta^k(X)$ with the disjoint union of these preimages yields a variety whose scheme-theoretic fiber over $P\in(\mathbb{C}^k)^{\circ}$ is precisely $\Delta^k_P(X)$. However, it is unclear whether one can take this partition to be independent of $X$, which would be necessary to extend the functoriality part of \Cref{mainthm} in this direction.

Motivated by examples like \Cref{2clumpexample}, we conjecture that one could simply stratify $(\mathbb{C}^k)^{\circ}$ according to the combinatorics of point-collisions in the incidence strata, i.e. which partial sums of the $m_i$ agree.
\begin{conj}
The stratification of $(\mathbb{C}^k)^{\circ}$ induced by the equivalence relation $(m_1,\ldots,m_k) \sim (n_1,\ldots,n_k)$ if $$\sum_{i \in A}m_i=\sum_{i \in B}m_i \Leftrightarrow \sum_{i \in A}n_i=\sum_{i \in B}n_i\text{ for all $A,B\subset \{1,\ldots,k\}$}$$
has the property that the preimage in $\Delta^k(X)$ of any equivalence class has reduced fibers over $(\mathbb{C}^k)^{\circ}$.
\end{conj}

\begin{comment}
\begin{rmk}
In all of our constructions, nothing prevents us from constructing a natural family $\Delta^k(\Spec A)$ for non-reduced $A$ as well. However, \Cref{mainthm} and \Cref{maincolthm} are not true with our constructions with scheme-theoretic fibers instead of reduced fibers. One can formally show for $A$ reduced that $\Delta^k(\Spec A)$ can be split into locally closed subvarieties where the scheme-theoretic fibers over $(\mathbb{C}^k)^{\circ}$ are correct (by being reduced). Regardless of the existence of a splitting for non-reduced $A$, it is unclear even in the reduced case whether such splittings could be taken independent of $A$ so that the resulting construction is functorial.
\begin{conj}
The splitting of $(\mathbb{C}^k)^{\circ}$ according to the equivalence relation on the power set of $\{1,\ldots,k\}$ where $A \sim A'$ if $\sum_{i \in A} m_i=\sum_{i \in A'}m_i'$ works.
\end{conj}
This would split according to the lattice of partial sums of the $m_i$, which governs the topology of the incidence stratum and determines the combinatorics of point collisions in the incidence stratum.
\end{rmk}
\end{comment}

\section{Interpolating colored incidence strata in $\mathbb{A}^n$}
\label{DeltaAn}
In this section we construct $\Delta^{k,r}(\mathbb{A}^n)$ as claimed in \Cref{maincolthm}. Our construction will be a generalization of \cite{Kasatani,Sergeev,Sam}. For $\overline{P}=(\overline{m_1},\ldots,\overline{m_k})\in (\mathbb{N}^r)^k$, we define $\overline{d}$ by
$$\overline{d}=(d_1,\ldots,d_r):=\overline{m_1}+\ldots+\overline{m_k},$$
where $\overline{P}$ is dropped from the notation.
\begin{defn}
Define the \emph{weighted elementary symmetric polynomial $e_n$} to be $$e_n(y_1,\ldots,y_k,x_1,\ldots,x_k):=\sum_{i_1+\ldots+i_k=n}\binom{y_1}{i_1}\ldots\binom{y_k}{i_k}x_1^{i_1}\ldots x_k^{i_k}$$
and define the formal series
%Really we should be working with $1-\frac{x_j}{z}$ instead of $z-x_j$ so that we have a genuine laurent series for arbitrary $d_i \in \mathbb{C}$, but this way is consistent with the descriptions of $\Sym^d \mathbb{A}^1$.

$$f_i(\overline{P},x_1,\ldots,x_k):=\prod_{j=1}^k(z-x_j)^{(\overline{m_j})_i}=\sum_{j=0}^\infty (-1)^jz^{d_i-j}e_j((\overline{m_1})_i,\ldots,(\overline{m_k})_i,x_1,\ldots,x_k).$$
The exponents of $z$ are complex numbers, but we introduce these expressions purely as a bookkeeping device. Define the weighted power sum $p_n$ to be
$$p_n(y_1,\ldots,y_k,x_1,\ldots,x_k):=y_1x_1^n+\ldots+y_kx_k^n.$$
We define $e_n(x_1,\ldots,x_k)$ and $p_n(x_1,\ldots,x_k)$ to be the usual elementary symmetric functions and power sums respectively, which are the specializations of the weighted versions to when all the weights are $1$.
\end{defn}
\begin{rmk} Note that for $\overline{P}=(\overline{m_1},\ldots,\overline{m_k})\in (\mathbb{N}^r)^k$, only finitely many coefficients of $f_i(\overline{P},x_1,\ldots,x_k)$ are non-zero, and in the terminology of the introduction, $$\Delta^k_{\overline{P}}(\mathbb{A}^1):=\{(f_1,\ldots,f_r)\mid x_1,\ldots,x_k \in \mathbb{C}\}\subset \Sym^{d_1}\mathbb{A}^1\times \ldots \times \Sym^{d_r}\mathbb{A}^1.$$
\end{rmk}
\Cref{gelflem} is given by restricting the Chow form $\Sym^d\mathbb{P}^n\to \mathbb{P}(H^0(\mathbb{P}^{n\vee},\mathscr{O}_{\mathbb{P}^{n\vee}(d)}))$ to an affine chart, where the Chow form maps the zero-cycle $P_1+\cdots+P_d$ to the product $L_1\cdots L_d$, where $L_i$ is the linear form on $\mathbb{P}^{n\vee}$ parameterizing hyperplanes through $P_i$.

\begin{lem}
\label{gelflem}
Let $V$ be an $n$-dimensional vector space. Then there is an embedding
\begin{align*}\Sym^d \mathbb{A}(V) &\hookrightarrow V \times \Sym^2V \times \cdots \times\Sym^d V\\
(v_1,\ldots,v_d)^{S_d} &\mapsto (-e_1(v_1,\ldots,v_k),e_2(v_1,\ldots,v_k),\ldots,(-1)^de_d(v_1,\ldots,v_k)).\end{align*}
\end{lem}
\begin{proof}
See \cite[Chapter 4, Proposition 2.1 and Theorem 2.2]{Gelfand}.
\end{proof}

\begin{rmk}
Confusingly, $\Sym^d \mathbb{A}(V)$ in \Cref{gelflem} is meant as the symmetric quotient $\mathbb{A}(V)^{d}/S_d$ as a variety, while $\Sym^d V$ is meant to be the vector space of dimension $\binom{n+d-1}{n-1}$. 
\end{rmk}

As a shorthand we may write this map as $$(v_1,\ldots,v_d)^{S_d} \mapsto (z-v_1)(z-v_2)\ldots(z-v_d)$$ where multiplication of the $v_i$'s is taken in the appropriate $\Sym^i V$. Then $$\Delta^{k,r}_{\overline{P}}(\mathbb{A}(V))\subset \Sym^{d_1}\mathbb{A}(V)\times \cdots \times \Sym^{d_r}\mathbb{A}(V)$$ for $\overline{m_i}\in(\mathbb{Z}^{\ge 0})^r\setminus \{\overline{0}\}$ corresponds to the locus
\begin{align*}\Delta^{k,r}_{\overline{P}}(\mathbb{A}(V))&=\{(f_1(\overline{P},v_1,\ldots,v_k),\ldots,f_r(\overline{P},v_1,\ldots,v_k)) \mid v_1,\ldots,v_k \in V\}\\
&\subset(V \times \cdots \times \Sym^{d_1}V)\times \cdots \times (V \times \cdots \times \Sym^{d_r}V)
\end{align*}
Our plan is to define $\Delta^{k,r}(\mathbb{A}(V))$ to be the image of \begin{align*}
    ((\mathbb{C}^r)^k)^{\circ}\times V^k &\to ((\mathbb{C}^r)^k)^{\circ}\times(\prod_{i=1}^\infty \Sym^iV)^r
\\
(\overline{m_1},\ldots,\overline{m_k},v_1,\ldots,v_k)&\mapsto (\overline{m_1},\ldots,\overline{m_k},f_1(\overline{P},v_1,\ldots,v_k),\ldots,f_r(\overline{P},v_1,\ldots,v_k)).
\end{align*}
with an appropriate scheme-structure placed on it. Here again $z$ is used for bookkeeping, and the map actually extracts the coefficients of the $f_i(\overline{P},v_1,\ldots,v_k)$.

Our choice to define $\Delta^{k,r}(\mathbb{A}(V))$ in this way is a priori surprising since we are working with a locus in an infinite-dimensional vector space and it is unclear what scheme structure we would like to put on it. However, we will show that there is an $N$ such that no information is lost after ignoring all $\Sym^{N'}V$ for $N' \ge N$, and after doing so the resulting map will be finite. From here $\Delta^{k,r}(\mathbb{A}(V))$ will be easily shown to have the desired properties in light of \Cref{gelflem}. The key technical lemmas adapted from \cite{Kasatani,Sergeev,Sam} are as follows.

\begin{lem}
\label{naklem}
    Suppose that $R$ is a reduced finite-type algebra over $\mathbb{C}$, and $\mathcal{A}\subset R[x_1,\ldots,x_\ell]$ is a sequence of positive degree homogenous polynomials such that if all polynomials in $\mathcal{A}$ vanish then all $x_i$ are equal to zero. Then $R[x_1,\ldots,x_\ell]$ is a finite $R[\mathcal{A}]$-module, and in particular $R[\mathcal{A}]$ is a finitely-generated $R$-algebra.
\end{lem}

\begin{proof}
%\cite[Lemma 1.4]{E05}
By applying the graded Nakayama's Lemma to the graded ring $R[\mathcal{A}]$ and the graded $R[\mathcal{A}]$-module $R[x_1,\ldots,x_\ell]$, it suffices to show that $R[x_1,\ldots,x_\ell]/(\mathcal{A})$ is a finite $R$-module. But the hypotheses and the Nullstellensatz imply that the radical of $(\mathcal{A})$ is precisely $(x_1,\ldots,x_\ell)$.
\end{proof}
\begin{lem}
\label{newtlem}
For every $N$ we have
$$\mathbb{C}[\{p_i(y_1,\ldots,y_k,x_1,\ldots,x_k)\}_{1 \le i \le N}]=\mathbb{C}[\{e_i(y_1,\ldots,y_k,x_1,\ldots,x_k)\}_{1 \le i \le N}].$$
\end{lem}
\begin{proof}
    The Newton identities hold in the weighted case when the weights are positive integers. Since $\mathbb{N}^k\subset \mathbb{C}^k$ is Zariski dense, they hold for all complex weights.
\end{proof}
\begin{lem}
\label{finitelem}
%Sorry about this, saying it correctly is annoying.
Let $\Spec R$ be either a point of $((\mathbb{C}^r)^k)^{\circ}$ or an affine open subset of $((\mathbb{C}^r)^k)^{\circ}$. Let $V\cong \mathbb{C}^n$ be an $n$-dimensional vector space, and $S$ be a polynomial ring over $R$ so that $\Spec S=\Spec R\times V^k$. Let $\mc{A}\subset S$ be the subset of polynomials given by pulling back the coordinate functions under
\begin{align*}
    \Spec R\times V^k\to \Sym^i V\\
    (\overline{m_1},\ldots,\overline{m_k},v_1,\ldots,v_k)&\mapsto e_i((\overline{m_1})_j,\ldots,(\overline{m_k})_j,v_1,\ldots,v_k)
\end{align*}
with $i \ge 1$ and $1 \le j \le n$. Then $S$ is a finite $R[\mathcal{A}]$-module and $R[\mathcal{A}]$ is a finitely generated $R$-algebra.
\end{lem}
\begin{proof}
We will show that $R$ and $\mathcal{A}$ satisfy the hypothesis of \Cref{naklem}. Suppose that for some choice of $(\overline{m_1},\ldots,\overline{m_k})\in \Spec R$ and $v_1,\ldots,v_k\in V$ that $e_i((\overline{m_1})_j,\ldots,(\overline{m_k})_j,v_1,\ldots,v_k)$ vanishes for all $i\geq 1,1\leq j\leq r$. \Cref{newtlem} then formally implies that for all such $i,j$ we have
$$(\overline{m_1})_jv_1^i+\ldots + (\overline{m_k})_jv_k^i=\overline{0}.$$ In particular, this implies that
$$\frac{1}{z-v_1}\overline{m_1}+\ldots +\frac{1}{z-v_k}\overline{m_k}-\frac{1}{z}\sum \overline{m_i}=0$$ by expanding in powers of $\frac{1}{z}$.
But by hypothesis on $R$ we have $\sum_{i \in A} \overline{m_i} \ne 0$ for all subsets $A$, so if some $v_i \ne 0$ it is impossible for there not to be a pole at $z=v_i$.
\end{proof}

Now, we construct $\Delta^{k,r}(\mathbb{A}(V))$ as follows. Cover $((\mathbb{C}^r)^k)^{\circ}$ with basic affine opens $\Spec A_1,\ldots \Spec A_w$ in $(\mathbb{C}^r)^k$. By \Cref{finitelem}, for $N$ sufficiently large the images of 
\begin{align*}
\Spec A_i\times V^k \to \Spec A_i \times (V\times \ldots \times \Sym^{N'} V)^r
\end{align*}
are all closed subvarieties for $N'\geq N$ and are isomorphic to each other by projection. Taking $N'=N$, these varieties also clearly glue together compatibly over the overlaps. This gives a natural scheme structure to the image of
\begin{align*}
    ((\mathbb{C}^r)^k)^{\circ} \times V^k \to ((\mathbb{C}^r)^k)^{\circ}\times(V\times \ldots \times \Sym^N V)^r.
\end{align*}
The reduced fiber over any point in $\overline{P}\in((\mathbb{C}^r)^k)^{\circ}$ is then precisely $\Delta^{k,r}_{\overline{P}}(\mathbb{A}(V))$.
\begin{rmk}
\label{finitermk}
By construction $$((\mathbb{C}^r)^k)^{\circ} \times V^k \to \Delta^{k,r}_{\overline{P}}(\mathbb{A}(V))$$ is finite, and hence in particular proper. Also, a similar application of \Cref{newtlem} shows that for fixed $\overline{m_1},\ldots,\overline{m_k}$, $(v_1,\ldots,v_k)$ and $(v_1',\ldots,v_k')$ correspond to the same point if and only if
$$\frac{1}{z-v_1}\overline{m_1}+\ldots+\frac{1}{z-v_k}\overline{m_k}=\frac{1}{z-v_1'}\overline{m_1}+\ldots+\frac{1}{z-v_k'}\overline{m_k}.$$
\end{rmk}

\section{Interpolating colored incidence strata of affine varieties}
\label{arbitrary}
In this section we prove \Cref{maincolthm} and discuss the Euclidean topology of our families of interpolated incidence strata. To each affine variety $X$, we embed $X \subset \mathbb{A}(V_X)$ for some vector space $V_X$, and then take $\Delta^{k,r}(X)$ to be the image of 
\begin{align}
((\mathbb{C}^r)^k)^{\circ}\times X^k \hookrightarrow ((\mathbb{C}^r)^k)^{\circ}\times \mathbb{A}(V_X)^k\to \Delta^{k,r}(\mathbb{A}(V_X)). \label{compositemap}
\end{align}

\begin{proof}[Proof of \Cref{maincolthm}]
Note that the map defining $\Delta^{k,r}(X)$ is the composite of two proper maps, and hence $\Delta^{k,r}(X)$ is a closed subvariety of $\Delta^{k,r}(\mathbb{A}(V_X))$.

We first verify that the fibers are correct at $\overline{P}\in (\mathbb{Z}_{\ge 0}^r\setminus \{\overline{0}\})^k$. To do this, note that the embedding $X \hookrightarrow \mathbb{A}(V_X)$ induces an embedding $\prod \Sym^{d_i}X \hookrightarrow \prod \Sym^{d_i} \mathbb{A}(V_X)$ which identifies $\prod \Sym^{d_i}X$ with the subset of $\prod \Sym^{d_i} \mathbb{A}(V_X)$ where the points are constrained to lie on $X$. Therefore, the incidence strata $\Delta^{k,r}_{\overline{P}}(X)\subset \Delta^{k,r}_{\overline{P}}(\mathbb{A}(V_X))$ is simply the subset of $\Delta^{k,r}_{\overline{P}}(\mathbb{A}(V_X))$ where the points are constrained to lie on $X$. Hence by construction the fiber over $\overline{P}\in (\mathbb{Z}_{\ge 0}^r\setminus \{\overline{0}\})^k$ is indeed $\Delta^{k,r}_{\overline{P}}(X)$ as desired.

Now we verify functoriality. Given a map $\phi:X \to Y$, we define the induced map $\Delta^{k,r}(X) \to \Delta^{k,r}(Y)$ to be at the level of sets
$$(\prod_{i=1}^k(z-\overline{x_i})^{(\overline{m_i})_1},\ldots,(\prod_{i=1}^k(z-\overline{x_i})^{(\overline{m_i})_r})\mapsto (\prod_{i=1}^k(z-\phi(\overline{x_i}))^{(\overline{m_i})_1},\ldots,(\prod_{i=1}^k(z-\phi(\overline{x_i}))^{(\overline{m_i})_r}).$$
First we check that this is a morphism. Indeed, suppose $\phi$ was induced from some identically named map $\phi:\mathbb{A}(V_X)\to \mathbb{A}(V_y)$. Then it suffices to show that we have an induced morphism $\Delta^{k,r}(\mathbb{A}(V_X))\to \Delta^{k,r}(\mathbb{A}(V_Y))$ since the image of $\Delta^{k,r}(X)$ clearly lies inside $\Delta^{k,r}(Y)$. Writing $\overline{x_i}=\sum_jr_{i,j}v_j$ for basis vectors $v_1,\ldots,v_n$ of $V_X$ and letting $w_1,\ldots,w_m$ be a basis of $V_Y$, it suffices by \Cref{newtlem} to show for each $\ell$ that we can write all of the coefficients of the coordinates in the $\prod w_i^{\mu_i}$ basis of $$ (\overline{m_1})_s\phi(\overline{x_1})^\ell+\ldots+(\overline{m_k})_s\phi(\overline{x_k})^\ell$$ as linear combinations of the coefficients of the coordinates of expressions in $$\{ (\overline{m_1})_s\overline{x_1}^i+\ldots+(\overline{m_k})_s\overline{x_k}^i\}_{1 \le i < \infty}.$$
The result follows from the following two observations.
\begin{itemize}
    \item Every coefficient of $ (\overline{m_1})_s\phi(\overline{x_1})^\ell+\ldots+(\overline{m_k})_s\phi(\overline{x_k})^\ell$ is a linear combination of elements of the form $$\sum_t(\overline{m_t})_s r_{t,1}^{\lambda_1}\ldots r_{t,n}^{\lambda_n}.$$ 
    \item The expression $\sum_t(\overline{m_t})_s r_{t,1}^{\lambda_1}\ldots r_{t,n}^{\lambda_n}$  is a multiple of the coefficient of $v_{1}^{\lambda_1}\ldots v_{n}^{\lambda_n}$ in $ (\overline{m_1})_s x_1^{\sum \lambda_i}+\ldots+(\overline{m_k})_s x_k^{\sum \lambda_i}$.
\end{itemize}
As the morphisms at the level of sets are manifestly functorial, it follows that the assignment is functorial as desired.
\end{proof}

\subsection{A topological model for weighted incidence strata.}
\label{topologicalmodel}
In this subsection, we will use the Euclidean topology instead of the Zariski topology. For a variety $X$ we may topologically define a family of weighted configuration spaces on $X$ as $(((\mathbb{C}^r)^k)^{\circ}\times X^k)/\sim$, formed as the quotient of the space of tuples  $(\overline{m_1},\ldots,\overline{m_k},x_1,\ldots,x_k)\in ((\mathbb{C}^r)^k)^{\circ}\times X^k$ by the relation that $$(\overline{m_1},\ldots,\overline{m_k},x_1,\ldots,x_k) \sim (\overline{m_1},\ldots,\overline{m_k},x_1',\ldots,x_k')$$ if for each $x \in X$ we have
$$\sum_{i\text{ such that }x_i=x}\overline{m_i}=\sum_{i\text{ such that }x_i'=x}\overline{m_i}.$$

\begin{prop}
\label{toppropAn}
The underlying topology of $\Delta^{k,r}(X)$ agrees with the family of weighted configuration spaces $(((\mathbb{C}^r)^k)^{\circ}\times X^k)/\sim$ above.
\end{prop}
\begin{proof}We first consider the case $X=\mathbb{A}(V)$.
Note that the condition $\sim$ is identical to the condition in \Cref{finitermk}. Thus the topologically proper map (\Cref{finitermk})
$$((\mathbb{C}^r)^k)^{\circ}\times V^k \to \Delta^{k,r}(\mathbb{A}(V))$$ factors through the bijective continuous
$$(((\mathbb{C}^r)^k)^{\circ}\times V^k)/\sim \to \Delta^{k,r}(\mathbb{A}(V)).$$ This implies the bijective continuous map is also proper, and hence a homeomorphism as desired.

For arbitrary $X$, because the composite map \eqref{compositemap} defining $\Delta^{k,r}(X)$ is proper, $\Delta^{k,r}(X)$ is topologically identified with the quotient $(((\mathbb{C}^r)^k)^{\circ}\times X^k)/\sim$.
\end{proof}

\section{The constant $N_k$}
\begin{defn}
\label{Nkdefn}
    Define
    $$R_i=\mathbb{C}[m_1,\ldots,m_k][\{\frac{1}{\sum_{i \in S}m_i}\}][\{e_j(m_1,\ldots,m_k,x_1,\ldots,x_k)\}_{1 \le j \le i}]$$ and for fixed $P=(m_1,\ldots,m_k) \in (\mathbb{C}^k)^{\circ}$, define
    $$R_i(P)=\mathbb{C}[\{e_j(m_1,\ldots,m_k,x_1,\ldots,x_k)\}_{1 \le j \le i}].$$ Define $N_k,N_k(P)$ to be the first points at which \begin{align*}R_{N_k}&=R_{N_k+1}=\ldots\\R_{N_k(P)}(P)&=R_{N_k(P)+1}(P)=\ldots\end{align*}
    
\end{defn}
Note $N_k$ and $N_k(P)$ exist by \Cref{finitelem}. Our first goal will be to show that $N_k(P)$ agrees with the definition from the introduction as the affine embedding dimension of $\Delta^k_P(\mathbb{A}^1)$.
We start with a useful proposition.
\begin{prop}
\label{stablem}
    If $R_i=R_{i+1}$, then $R_i=R_{i+1}=R_{i+2}=\ldots$. Also, if $R_i(P)=R_{i+1}(P)$ then $R_i(P)=R_{i+1}(P)=R_{i+2}(P)=\ldots$.
\end{prop}
\begin{proof}
    Note that \Cref{newtlem} allows us to consider the $p_i$'s stabilizing rather than the $e_i$'s. Let $p_j=p_j(m_1,\ldots,m_k,x_1,\ldots,x_k)$, and suppose we have an equation
    $$p_{i+1}=g(p_1,\ldots,p_i)$$ where $g$ is a polynomial with coefficients in $\mathbb{C}[m_1,\ldots,m_k][\{\frac{1}{\sum_{i \in S}m_i}\}]$ if we're working with $R_i$, or $\mathbb{C}$ if we're working with $R_i(P)$. In either case, substituting $x_u+\epsilon x_u^t$ into $x_u$ for every $1\le u \le k$ yields
    $$p_{i+1}+\epsilon (i+1)p_{t+i}\equiv g(p_1,\ldots,p_i)+\epsilon\sum_{j=1}^i (\partial_jg)(p_1,\ldots,p_i)jp_{t+j -1}\pmod{\epsilon^2}.$$
    Equating the $\epsilon$ terms then yields
    $$p_{t+i}=\frac{1}{i+1}\sum_{j=1}^i (\partial_jg)(p_1,\ldots,p_i)jp_{t+j-1}.$$
\end{proof}
Thus, to find $N_k$ or $N_k(P)$ it suffices to find the first point at which the corresponding rings start stabilizing.
\begin{lem}
\label{mingenlem}
If $f_1,f_2,\ldots,f_u$ is a sequence of homogenous polynomials in $\mathbb{C}[x_1,\ldots,x_\ell]$, such that no $f_i$ is redundant as a generator of $\mathbb{C}[f_1,\ldots,f_u]$, then $u$ is the minimum number of generators of $\mathbb{C}[f_1,\ldots,f_u]$ as a $\mathbb{C}$-algebra.
\end{lem}
\begin{proof}
Omitted.
\end{proof}
\begin{prop}
Given $P=(m_1,\ldots,m_k)\in (\mathbb{C}^k)^{\circ}$, $N_k(P)$ from \Cref{Nkdefn} equals
$$\min \{n \mid \Delta^k_P(\mathbb{A}^1)\text{ embeds into } \mathbb{A}^n\}.$$
\end{prop}
\begin{proof}
It is clear from \Cref{stablem} that $\{e_i(m_1,\ldots,m_k,x_1,\ldots,x_k)\}_{1 \le i \le N_k(P)}$ satisfies the hypothesis of \Cref{mingenlem}.
\end{proof}

We now show that $k$-part $r$-colored incidence strata of varieties $X \subset \mathbb{A}^n$ are embeddedable in $\mathbb{A}^{r(\binom{n+N_k}{n}-1)}$.

\begin{prop}\label{smallembedprop}
   Given a subvariety $X \subset \mathbb{A}^n$ and $\overline{P}\in ((\mathbb{C}^r)^k)^{\circ}$, there is an embedding
    $$\Delta^{k,r}_{\overline{P}}(X)\subset \mathbb{A}^{r(\binom{n+N_k}{n}-1)}.$$
\end{prop}
\begin{proof}
    Write $\mathbb{A}^n=\mathbb{A}(V)$, and let $R=\Gamma(\Delta^{k,r}_{\overline{P}}(\mathbb{A}(V)))$. Since by construction $\Delta^{k,r}_{\overline{P}}(X)\subset \Delta^{k,r}_{\overline{P}}(\mathbb{A}(V)),$ it suffices to show that $$\Delta^{k,r}_{\overline{P}}(\mathbb{A}(V))\subset \mathbb{A}^{r(\binom{n+N_k}{n}-1)}.$$ By construction, we have
    $R$ is generated as a $\mathbb{C}$-algebra by the coordinate projections of all maps of the form
    $$(v_1,\ldots,v_k)\mapsto e_i((\overline{m_1})_j,\ldots,(\overline{m_k})_j,v_1,\ldots,v_k)\in \Sym^iV.$$
    \Cref{newtlem} formally implies that $R$ is also generated by the coordinate projections of all maps of the form
    $$(v_1,\ldots,v_k)\mapsto (\overline{m_1})_jv_1^i+\ldots+(\overline{m_k})_jv_k^i \in \Sym^iV.$$
    Written in this form it is clear that applying a common invertible linear transformation to all $\overline{m_i}$'s maps the algebra to precisely the same algebra, so we may apply such a map to assume without loss of generality that for every $j$ and $S\subset \{1,\ldots,k\}$ that 
    $$\sum_{i \in S}(\overline{m_i})_j \ne 0.$$
    
    But for each $j$ and $i > N_k$, we obtain by \Cref{stablem} a polynomial relation
    \begin{align*}e&_i((\overline{m_1})_j,\ldots,(\overline{m_k})_j,v_1,\ldots,v_k)\\
    &=g_i(e_1((\overline{m_1})_j,\ldots,(\overline{m_k})_j,v_1,\ldots,v_k),\ldots,e_{i-1}((\overline{m_1})_j,\ldots,(\overline{m_k})_j,v_1,\ldots,v_k)).
    \end{align*}
    Thus we only need to use polynomials up to $e_{N_k}$ to generate $R$ and the result now follows by counting the number of coordinate projections of such functions.
\end{proof}
\begin{rmk}
    For the case $r=1$ we could have used $N_k(P)$ instead of $N_k$.
    \end{rmk}
\begin{rmk}Following the proof of \Cref{smallembedprop}, it is clear that $\Delta^{k,r}_{\overline{P}}(X)$ is unchanged if we act by $GL_r$ on $\overline{P}$.
\end{rmk}

\begin{prop}
We have $N_2((1,2))=N_2=3$.
\end{prop}
\label{N2}
\begin{proof}
    $N_2 \le 3$ by \Cref{newtlem}, \Cref{stablem} and the relation
    $$(m_1+m_2)p_4-4p_1p_3+3p_2^2=((m_1+m_2)p_2-p_1^2)^2.$$
    $N_2 \ge N_{2}((1,2))=3$ since taking $m_1=1$ and $m_2=2$, it is easy to see that there is no way of expressing $x_1^3+2x_2^3$ as a polynomial in $x_1+2x_2$ and $x_1^2+2x_2^2$, and we conclude by \cref{newtlem}.
\end{proof}
\begin{prop}
\label{isotrivprop}
$\Delta^{2}(\mathbb{A}^n)$ is isomorphic over $(\mathbb{C}^2)^{\circ}$ to the image of
$$(\mathbb{C}^2)^{\circ}\times\mathbb{A}^n\times \mathbb{A}^n \to (\mathbb{C}^2)^{\circ}\times \mathbb{C}^n \times \Sym^2\mathbb{C}^n \times \Sym^3\mathbb{C}^n$$
taking
$$(m_1,m_2,v_1,v_2)\mapsto (m_1,m_2,v_1,v_2^2,(m_1-m_2)v_2^3).$$
In particular, the family is isotrivial.
\end{prop}
\begin{proof}
By \Cref{N2} and \Cref{newtlem}, $\Delta^2(\mathbb{A}^n)$ is isomorphic to the image of the map
$$(m_1,m_2,v_1,v_2)\mapsto(m_1,m_2,p_1(m_1,m_2,v_1,v_2),p_2(m_1,m_2,v_1,v_2),p_3(m_1,m_2,v_1,v_2)).$$
We first do a polynomial change of coordinates on the codomain to obtain the new map
$$(m_1,m_2,v_1,v_2)\mapsto (m_1,m_2,m_1v_1+m_2v_2,(v_1-v_2)^2,(m_1-m_2)(v_1-v_2)^3).$$
Indeed, observe that we have the following universal relations between $p_i(m_1,m_2,x_1,x_2)$ with $x_1,x_2$ indeterminates:
\begin{align*}
    (x_1-x_2)^2&=\frac{1}{m_1m_2}((m_1+m_2)p_2-p_1^2)\\
    (m_1-m_2)(x_1-x_2)^3&=-\frac{1}{m_1m_2}\left((m_1+m_2)^2p_3-3(m_1+m_2)p_1p_2+2p_1^3\right).
\end{align*}
Now, we may change coordinates on the domain with $v_1'=mv_1+nv_2$ and $v_2'=v_1-v_2$ (which is an invertible linear transformation as $m+n \ne 0$).
\end{proof}
\begin{rmk}
In \cite{Abdesselam1,Abdesselam2}, such $2$-part incidence strata were studied in $\mathbb{P}^n$ rather than $\mathbb{A}^n$. One can show along the lines of \Cref{notisotrivprop} that $2$-part incidence strata in $\mathbb{P}^n$ are in general not isomorphic to each other, despite being locally isomorphic by \Cref{isotrivprop}.
\end{rmk}
\begin{rmk}
It would be tempting to try to conclude that there are only two distinct isomorphism classes of two-part incidence strata on \emph{any} affine variety $V$ by its embedding into $\mathbb{A}^n$, however as the final change of coordinates in the proof is not an automorphism of $V\times V \subset \mathbb{A}^n \times \mathbb{A}^n$, we are not allowed to apply it.
\end{rmk}
\begin{rmk}
For $n=1$ we see that all two-part incidence strata on $\mathbb{A}^1$ are equal to $\Sym^2\mathbb{A}^1\cong \mathbb{A}^2$ if the parts are the same, and $\mathbb{A}^1 \times C$ if the parts are unequal, where $C$ is the cuspidal cubic $y^2=x^3$ in $\mathbb{A}^2$.
\end{rmk}
\begin{prop}
\label{notisotrivprop}
The family $\Delta^3(\mathbb{A}^n)$ is not iso-trivial over $(\mathbb{C}^3)^{\circ}$. In fact, if $m_1,m_2,m_3\in \mathbb{N}$ are distinct and $m_1',m_2',m_3'\in \mathbb{N}$ are distinct, such that $(m_1,m_2,m_3)$ is not a multiple of $(m_1',m_2',m_3')$ after possibly rearranging the coordinates of $(m_1,m_2,m_3)$, then $\Delta^3_{(m_1,m_2,m_3)}(\mathbb{A}^n) \not\cong \Delta^3_{(m_1',m_2',m_3')}(\mathbb{A}^n)$.
\end{prop}
\begin{proof}
Suppose that there was an isomorphism. Then there is an isomorphism of normalizations of $\Delta^3_{(m_1,m_2,m_3)}(\mathbb{A}^n)$ and $\Delta^3_{(m_1',m_2',m_3')}(\mathbb{A}^n)$ fitting into the diagram
\begin{center}
\begin{tikzcd}
\mathbb{A}^n\times\mathbb{A}^n\ar[r]\ar[d]\times\mathbb{A}^n&\Delta^n_{(m_1,m_2,m_3)}(\mathbb{A}^n)\ar[d]\\
\mathbb{A}^n\times\mathbb{A}^n\times\mathbb{A}^n\ar[r]&\Delta^n_{(m_1',m_2',m_3')}(\mathbb{A}^n).
\end{tikzcd}
\end{center}
For either incidence strata, the singular locus occurs when a point collision occurs, and the singular locus of the singular locus is when all 3 points collide. This corresponds to the image of the diagonal in $\mathbb{A}^n\times \mathbb{A}^n \times \mathbb{A}^n$, so the left vertical isomorphism $(v_1,v_2,v_2)\mapsto (f(v_1,v_2,v_3),g(v_1,v_2,v_3),h(v_1,v_2,v_3))$ must have the constant terms of $f,g,h$ all equal. Applying the translation automorphism to the $\mathbb{A}^n$ that $\Delta^3_{m_1',m_2',m_3'}(\mathbb{A}^n)$ is the incidence strata for, we may assume that $f(0)=g(0)=h(0)=0$. Letting $f',g',h'$ represent the inverse isomorphism, the constant terms of $f',g',h'$ are then also zero.

Now, the isomorphism implies that the maps
\begin{align*}&(z-v_1)^{m_1}(z-v_2)^{m_2}(z-v_3)^{m_3}\mapsto\\ &(z-f(v_1,v_2,v_3))^{m_1'}(z-g(v_1,v_2,v_3))^{m_2'}(z-h(v_1,v_2,v_3))^{m_3'},\text{ and }\\
&(z-v_1)^{m_1'}(z-v_2)^{m_2'}(z-v_3)^{m_3'}\mapsto\\
&(z-f'(v_1,v_2,v_3))^{m_1}(z-g'(v_1,v_2,v_3))^{m_2}(z-h'(v_1,v_2,v_3))^{m_3}
\end{align*}
are both morphisms.
Because the coefficients are all homogenous expressions in the coordinates of the coefficients of $(z-v_1)^{m_1}(z-v_2)^{m_2}(z-v_3)^{m_3}$, replacing each of $f,g,h$ with their lowest order homogenous terms (i.e. their linear terms since $(f,g,h)$ is an automorphism of $\mathbb{A}^n\times \mathbb{A}^n \times \mathbb{A}^n$) results in a morphism $\Delta^3_{(m_1,m_2,m_3)}(\mathbb{A}^n)\to \Delta^3_{(m_1',m_2',m_3')}(\mathbb{A}^n)$, and similarly replacing each of $f',g',h'$ with their lowest order terms yields an inverse morphism $\Delta^3_{(m_1',m_2',m_3')}(\mathbb{A}^n)\to \Delta^3_{(m_1,m_2,m_3)}(\mathbb{A}^n)$.
Thus, we may assume that $f,g,h$ are linear.

Now, as the singular locus occurs during a point collision, we must have that $v_i=v_j$ implies two of $f(v_1,v_2,v_3), g(v_1,v_2,v_3), h(v_1,v_2,v_3)$ are equal. Since $f,g,h$ are distinct linear forms, this implies that possibly after reordering $f,g,h$ (which we are allowed to do by reordering $m_1',m_2',m_3'$), there exist linear $R,L$ such that \begin{align*}
f(v_1,v_2,v_3)&=R(v_1,v_2,v_3)+L(v_1)\\ g(v_1,v_2,v_3)&=R(v_1,v_2,v_3)+L(v_2)\\ h(v_1,v_2,v_3)&=R(v_1,v_2,v_3)+L(v_3)
\end{align*}
As $(f,g,h)$ is a full rank linear map, $(f,g-f,h-f)$ is also a full rank linear map, so $L$ must be invertible. Let $w_1$ be an eigenvector for $L$, so $L(w_1)=\lambda w_1$ for some $\lambda \ne 0$. Extend $w_1$ to a basis $w_1,\ldots,w_n$ of $\mathbb{C}^n$ and write $v_i=\sum_j r_{i,j}w_j$. In this basis, the coordinate projections of the expression
\begin{align*}(m'+n'+p')p_2(f(v_1,v_2,v_3),g(v_1,v_2,v_3),h(v_1,v_2,v_3))\\-p_1(f(v_1,v_2,v_3),g(v_1,v_2,v_3),h(v_1,v_2,v_3))^2=\\
m_1'm_2'(L(v_1)-L(v_2))^2+m_1'm_3'(L(v_1)-L(v_3))^2+m_2'm_3'(L(v_2)-L(v_3))^2
\end{align*}
must be expressible as polynomials in the coordinate projections of $p_1(v_1,v_2,v_3)$ and $p_2(v_1,v_2,v_3)$. Setting all $r_{i,j}$ to be zero except for $x=r_{1,1}$, $y=r_{2,1}$, $z=r_{3,1}$, we see that there must exist $\lambda_1,\lambda_2$ such that $$m'n'(x-y)^2+m'p'(x-z)^2+n'p'(y-z)^2=\lambda_1(mx+ny+pz)^2+\lambda_2(mx^2+ny^2+pz^2).$$
For the right hand side to be zero when $x=y=z$ we need $\lambda_2=-\lambda_1$, and get
$$m'n'(x-y)^2+m'p'(x-z)^2+n'p'(y-z)^2=\lambda_1(mn(x-y)^2+np(x-z)^2+np(y-z)^2).$$ But this is only possible if $m'n',m'p',n'p'$ are proportional to $mn,mp,np$, which implies that $m',n',p'$ are constant multiples of $m,n,p$, a contradiction.
\end{proof}
We now show \Cref{HuntersConjecture} for $k \le 3$ and show $N_4((1,2,4,8))\ge 15$.
\begin{proof}[Proof of \Cref{Conjthm}]
These results were found by Magma computations \cite{Magma}. To show $N_3((1,2,4))\ge 7$ we show that there are no linear relations between degree $7$ products of polynomials of the form $x^i+2y^i+4z^i$, and to show $N_4 \ge 15$ we show that there are no linear relations between degree $15$ products of polynomials of the form $x^i+2y^i+4z^i+8w^i$. To show $N_3=7$, we used Magma to find all linear relations over $\mathbb{C}(m_1,m_2,m_3)$ between degree $15$ products of polynomials of the form $p_i(m_1,m_2,m_3,x_1,x_2,x_3)$ (the degree of $p_i$ being $i$). There is in fact a unique linear relation up to scaling, and after clearing the denominators the expression was of the form
$$(m_1+m_2+m_3)(m_1+m_2)(m_1+m_3)(m_2+m_3)m_1m_2m_3p_8=g(m_1,m_2,m_3,p_1,\ldots,p_7)$$
for some polynomial $g$. As $(m_1+m_2+m_3)(m_1+m_2)(m_1+m_3)(m_2+m_3)m_1m_2m_3$ is invertible we have a relation showing $R_7=R_8$, and hence \Cref{stablem} implies $N_3 \ge 7$. Thus $N_3=7$.
\end{proof}
\section{Explicit recursions for weighted power sums}
\label{explicitsect}
In this section, we sketch a construction of an explicit polynomial expressing $p_N$ in terms of the previous $p_i$ (which by \Cref{newtlem} is equivalent to relating $e_N$ to the previous $e_i$).

First, we will construct expressions realizing some power of $x_1$ to be inside the ideal $(p_1,\ldots,)\subset \mathbb{C}[m_1,\ldots,m_k,x_1,\ldots,x_k][\{\frac{1}{\sum_{i \in S}m_i}\}_S]$ by induction on $k$. %For brevity, we denote $$\ldots x_k := m_1,\ldots,m_k,x_1,\ldots,x_k.$$ 
Supposing we have an expression for $k$ variables 
$$x_1^{w_k}-g_1(\{m_i,x_i\}_{i=1}^{k})p_1(\{m_i,x_i\}_{i=1}^{k})+\ldots+g_{w_k}(\{m_i,x_i\}_{i=1}^{k})p_{w_k}(\{m_i,x_i\}_{i=1}^{k})=0,$$ we show how to construct a relation in $k+1$-variables
$$x_1^{w_{k+1}}-g_1'(\{m_i,x_i\}_{i=1}^{k+1})p_1(\{m_i,x_i\}_{i=1}^{k+1})+\ldots+g'_{w_{k+1}}(\{m_i,x_i\}_{i=1}^{k+1})p_{w_{k+1}}(\{m_i,x_i\}_{i=1}^{k+1})=0.$$

From the relation (denoting $p_i$ for $p_i(\{m_i,x_i\}_{i=1}^{k+1})$) $$\prod_{1 \le i < j \le k+1}(x_i-x_j)^2=\frac{1}{m_1\ldots m_{k+1}}\begin{bmatrix}p_0 & p_1 & \ldots & p_{k}\\ p_1 & p_2 & \ldots & p_{k+1}\\
\ldots & \ldots & \ldots&\ldots\\
p_{k}&p_{k+1}&\ldots&p_{2k}\end{bmatrix},$$
we see that such a relation in $k+1$ variables is equivalent to producing an expression
$$x_1^{w_{k+1}}-g_1'(\{m_i,x_i\}_{i=1}^{k+1})p_1(\{m_i,x_i\}_{i=1}^{k+1})+\ldots+g'_{w_{k+1}}(\{m_i,x_i\}_{i=1}^{k+1})p_{w_{k+1}}(\{m_i,x_i\}_{i=1}^{k+1})$$
divisible by $\prod_{1 \le i < j \le k+1}(x_i-x_j)^2$. To do this it suffices to find for every $1 \le i < j \le k+1$ such an expression divisible by $x_i-x_j$ as then we can multiply together the squares of all such expressions. So suppose we fix $1 \le i < j \le k+1$. We obtain such an expression for $k+1$ variables from the $k$-variable relation by replacing $p_\ell(\{m_i,x_i\}^k_{i=1})$ with $p_\ell(\{m_i,x_i\}^{k+1}_{i=1})$, replacing in $g_\ell$ the variables $m_1,\ldots,m_k$ with
$$m_1,\ldots,m_{i-1},m_i+m_j,m_{i+1},\ldots,m_{j-1},m_{j+1},\ldots,m_{k+1}$$
and replacing the variables
$x_1,\ldots,x_k$ in $g_\ell$ with $$x_1,\ldots,x_{j-1},x_{j+1},\ldots,x_{k+1}.$$
Indeed, if we let $x_i=x_j$ in such an expression, we obtain the $k$-variable expression with the $m_\ell$ and $x_\ell$ replaced with the above sets of variables. Thus the induction is complete. By permuting the variables we immediately obtain similar expressions for $x_2^{w_{k}},\ldots,x_{k+1}^{w_{k}}$.

By the above relations, we may now set up $w_k^k$ linear equations expressing $x_1$ times a monomial $x_1^{i_1}\ldots x_{k}^{i_{k}}$ with $0 \le i_1,\ldots,i_{k} < w_{k}$ as a linear combination of other such monomials with each coefficient a multiple of some $p_i$. Expressing this system of equations as a matrix, the Cayley-Hamilton theorem \cite[Theorem 10.2]{18705} then says that $x_1$ satisfies the characteristic polynomial of this matrix, yielding a monic polynomial in $x_1$ with coefficients in the algebra $$R=\mathbb{C}[m_1,\ldots,m_k][\{\frac{1}{\sum_{i \in S}m_i}\}_S][p_1,p_2,\ldots].$$ We similarly obtain such polynomials for $x_2,\ldots,x_k$, and multiplying these $k$ polynomials together yields a single monic polynomial $h$ with coefficients in $R$ satisfied by $x_1,\ldots,x_k$. The expression $m_1x_1h(x_1)+m_2x_2h(x_2)+\ldots+m_kx_kh(x_k)$ is then manifestly an expression for $p_{kw_{k}^k+1}$ in terms of the previous $p_i$.

%RESUME HERE
This construction yields a polynomial of degree roughly $k!^{2k}$. If we had used the effective Nullstellensatz \cite[Theorem 1.3]{Jelonek} at the first stage to bound $w_k$ it would have given a bound of approximately $k!^k$ for $N_k(P)$ and $2^{k^2}k!^k$ for $N_k$, though the construction would be non-explicit.

\section{Singularities of colored incidence strata in smooth curves}
\label{singsect}
In this section, we prove \Cref{singthm}.

\begin{proof}[Proof of \Cref{singthm}]
Let $\overline{P}=(\overline{m_1},\ldots\overline{m_k})$, and denote
$$\overline{m_1}+\ldots+\overline{m_k}=(d_1,\ldots,d_r).$$ First we prove the result for $\CC=\mathbb{A}^1$.

Let $\overline{n_1},\ldots,\overline{n_t}$ denote the distinct $\overline{m_i}$, and suppose $\overline{n_i}$ appears $\ell_i$ times in $\overline{P}$. Then the normalization of $\Delta^{k,r}_{\overline{P}}(\mathbb{A}^1)$ is the morphism
\begin{align*}\Phi_{\overline{P}}:\Sym^{\ell_1}\mathbb{A}^1\times \ldots \times \Sym^{\ell_t}\mathbb{A}^1 &\to (\mathbb{A}^\infty)^r
\\
    (F_1,\ldots, F_t) &\mapsto (\prod_{s=1}^tF_s^{(\overline{n_s})_1},\ldots,\prod_{s=1}^tF_s^{(\overline{n_s})_r}).
\end{align*}
Indeed, by \Cref{finitelem}, the inclusion $\Gamma(\Delta^{k,r}_{\overline{P}}(\mathbb{A}^1))\subset \mathbb{C}[x_1,\ldots,x_k]$ is finite, and the map $\Phi_{\overline{P}}$ corresponds to the inclusion
$$\Gamma(\Delta^{k,r}_{\overline{P}}(\mathbb{A}^1))\subset \mathbb{C}[x_1,\ldots,x_k]^{S_{\ell_1}\times \ldots \times S_{\ell_t}},$$
which is thus also finite. It is furthermore birational as it is generically injective, and the domain is normal, so it is the normalization as desired.

Branches at $\overline{X}$ are in bijection with points of $\Phi_P^{-1}(\overline{X})$, so the description of the branches immediately follows. Now we check when $\phi_{\overline{P}}$ is an immersion at $\overline{Y}$, i.e. when $(d\Phi_P)_{\overline{Y}}$ is not of full rank.

Suppose $$Y=(F_1(z),\ldots,F_t(z))\in \Sym^{\ell_1}\mathbb{A}^1\times \ldots \times \Sym^{\ell_t}\mathbb{A}^1,$$ where $F_i=z^{\ell_i}+w_{i,1}z^{\ell_i-1}+\ldots+w_{i,\ell_i}$ is a degree $\ell_i$ monic polynomial in $z$. By taking the derivative with respect to each of the coefficients of each of the $F_i$ separately, the Jacobian of $\Phi_{\overline{P}}$ is seen to be given by the block matrix
$\begin{bmatrix}M_{i,j}\end{bmatrix}^T$ with $1 \le i \le t$, $1\le j \le r$,
where $M_{i,j}$ is the $\ell_i\times \infty$ matrix below (whose columns are indexed by powers of $z$)\footnote{Here we use the fact that taking the derivative with respect to a coefficient of an $F_i$ can be computed formally using the chain rule rather than expanding the series and computing the derivative of each coefficient of the resulting series separately.}. We will by abuse of notation use power series and infinite row vectors interchangeably, where the correspondence in one direction is given by listing the coefficients.
\begin{align*}M_{i,j}&=\begin{bmatrix}\partial/\partial w_{i,1}\prod_{s=1}^tF_s^{(\overline{n_s})_j}\\\ldots\\\partial/\partial w_{i,\ell_i}\prod_{s=1}^tF_s^{(\overline{n_s})_j}\end{bmatrix}=(\overline{n_i})_j\begin{bmatrix}z^{\ell_i-1}\left(F_i^{(\overline{n_i})_j-1}\prod_{s\ne i}F_s^{(\overline{n_s})_j}\right)\\z^{\ell_i-2}\left(F_i^{(\overline{n_i})_j-1}\prod_{s\ne i}F_s^{(\overline{n_s})_j}\right)\\ \dots \\ 1\cdot \left(F_i^{(\overline{n_i})_j-1}\prod_{s\ne i}F_s^{(\overline{n_s})_j}\right)\end{bmatrix}\\
&=(\overline{n_i})_j\begin{bmatrix}z^{\ell_i-1}(z^{d_i-\ell_i}+a_{i,1}z^{d_i-\ell_i-1}+\ldots)\\z^{\ell_i-2}(z^{d_i-\ell_i}+a_{i,1}z^{d_i-\ell_i-1}+\ldots)\\ \dots \\ 1\cdot (z^{d_i-\ell_i}+a_{i,1}z^{d_i-\ell_i-1}+\ldots)\end{bmatrix}\\
&=(\overline{n_i})_j
\begin{bmatrix}z^{d_i-1}&a_{i,1}z^{d_i-2}&a_{i,2}z^{d_i-3}&\ldots& a_{i,\ell_i-1}z^{d_i-\ell_i}&a_{i,\ell_i}z^{d_i-\ell_i-1}&\ldots\\0 &  z^{d_i-2}&a_{i,1}z^{d_i-3}&\ldots&a_{i,\ell_i-2}z^{d_i-\ell_i}& a_{i,\ell_i-1}z^{d_i-\ell_i-1}&\ldots\\ 0&0&z^{d_i-3}&\ldots&a_{i,\ell_i-3}z^{d_i-\ell_i}& a_{i,\ell_i-2}z^{d_i-\ell_i-1}&\ldots\\\ldots & \ldots& \ldots& \ldots&\ldots&\ldots&\ldots\\0 & 0 &\ldots&\ldots & z^{d_i-\ell_i}& a_{i,1}z^{d_i-\ell_i-1}&\ldots \end{bmatrix}.
\end{align*}
The rank of $d\Phi_{\overline{P}}$ is the same as the row rank of $[M_{i,j}]$. Multiplying each row of $M_{i,j}$ by $\left(\prod_{s=1}^tF_s^{(\overline{n_s})_j}\right)^{-1}$ is equivalent to right multiplication by an invertible (and upper triangular) $\infty\times\infty$ matrix (independent of $i$). Doing so does not change the $\mathbb{C}$-linear dependencies between the rows of the block matrix $[M_{i,j}]$. We see that
%We may scale the $j$'th block of columns by $\left(\prod_{s=1}^tF_s^{(\overline{n_s})_j}\right)^{-1}$ without changing the linear dependencies between the rows of the $[M_{i,j}]$ (as this is an invertible linear transformation with inverse given by multiplying by $\prod_{s=1}^tF_s^{(\overline{n_s})_j}$). This scaling yields
\begin{align*}
    &\left(\prod_{s=1}^tF_s^{(\overline{n_s})_j}\right)^{-1}M_{i,j}=(\overline{n_i})_j\begin{bmatrix}z^{\ell_i-1}/F_i\\z^{\ell_i-2}/F_i\\ \ldots \\ 1/F_i\end{bmatrix}\\
    &=(\overline{n_i})_j\begin{bmatrix}z^{\ell_i-1}(z^{-\ell_i}+b_{i,1}z^{-\ell_i-1}+\ldots)\\z^{\ell_i-2}(z^{-\ell_i}+b_{i,1}z^{-\ell_i-1}+\ldots)\\ \dots \\ 1\cdot(z^{-\ell_i}+b_{i,1}z^{-\ell_i-1}+\ldots)\end{bmatrix}\\
&=(\overline{n_i})_j
\begin{bmatrix}z^{-1}&b_{i,1}z^{-2}&b_{i,2}z^{-3}&\ldots& b_{i,\ell_i-1}z^{-\ell_i}&b_{i,\ell_i}z^{-\ell_i-1}&\ldots\\0 &  z^{-2}&b_{i,1}z^{-3}&\ldots&b_{i,\ell_i-2}z^{-\ell_i}& b_{i,\ell_i-1}z^{-\ell_i-1}&\ldots\\ 0&0&z^{d_i-3}&\ldots&b_{i,\ell_i-3}z^{-\ell_i}& b_{i,\ell_i-2}z^{-\ell_i-1}&\ldots\\\ldots & \ldots& \ldots& \ldots&\ldots&\ldots&\ldots\\0 & 0 &\ldots&\ldots & z^{-\ell_i}& b_{i,1}z^{-\ell_i-1}&\ldots \end{bmatrix}
\end{align*}
Hence the rows of the block matrix $\begin{bmatrix}M_{i,j}\end{bmatrix}$ are linearly dependent if and only if there is a $\mathbb{C}$-linear dependence between the vectors
$$\frac{z^{\ell_i-1}}{F_1}\overline{n_1},\frac{z^{\ell_1-2}}{F_1}\overline{n_1},\ldots,\frac{1}{F_1}\overline{n_1},\ldots,\frac{z^{\ell_t-1}}{F_t}\overline{n_t},\frac{z^{\ell_t-2}}{F_t}\overline{n_t},\ldots,\frac{1}{F_t}\overline{n_t},$$ i.e. there exist polynomials $G_i$ with $\deg(G_i)<\deg(F_i)$ and $\sum \frac{G_i}{F_i}\overline{n_i}=0$. Denote by $T$ the set of roots of all $F_i$. By partial fraction decomposition, we can write $$\{\frac{G_i}{F_i} \mid \deg(G_i)<\deg(F_i)\}=\bigoplus_{\substack{\alpha \in T, j \ge 1\\(z-\alpha)^j \mid f_i}}\mathbb{C}\frac{1}{(z-\alpha)^j},$$
and hence there is a linear dependency of this form if and only if there exists an $\alpha \in T$ and $j \ge 1$ such that $\{\overline{n_i} \mid (z-\alpha)^j\text{ divides }f_i\}$ has a linear dependency. Obviously if this occurs then it occurs when $j=1$ for some $\alpha \in T$, and the result now follows for $\Delta^{k,r}_{\overline{P}}(\mathbb{A}^1)$.

Now, let $\CC$ be an arbitrary smooth affine curve. By passing to an open subset we may assume that there is an \'etale morphism $$\Psi:\CC\to \mathbb{A}^1,$$ inducing a map $$\widetilde{\Psi}:\Delta^{k,r}_{\overline{P}}(\CC) \to \Delta^{k,r}_{\overline{P}}(\mathbb{A}^1).$$ We first show the result for $\overline{m_i} \in \mathbb{N}^k$. To do this it suffices to show that the formal completion of $\overline{X}\in \Delta^{k,r}_{\overline{P}}(\CC)$ maps isomorphically to the formal completion of $\widetilde{\Psi}(\overline{X})$. Indeed, in this case we may represent $\Delta^{k,r}_{\overline{P}}(\CC)$ as the quotient of a certain ordered colored incidence strata by $S_{d_1}\times \ldots \times S_{d_r}$. Let $$\Delta^{ord}_{\overline{P}}(\CC)\subset\CC^{d_1}\times\ldots \times \CC^{d_r}$$ be all ordered incidences of $d_i$ distinguishable points of color $i$ for each $i$ which correspond to an element of  $\Delta^{k,r}_{\overline{P}}(C)$, then $\Delta^{k,r}_{\overline{P}}(\CC)=\Delta^{ord}_{\overline{P}}(\CC)/S_{d_1}\times \ldots \times S_{d_r}$. Note that $\Delta^{ord}_{\overline{P}}(\CC)$ is simply the union of many diagonally embedded copies of $\CC^k$ encoding which distinguished points are incident. It is also clear that the map
$$\Delta^{ord}_{\overline{P}}(\CC) \to \Delta^{ord}_{\overline{P}}(\mathbb{A}^1)$$ induces an $S_{d_1}\times \ldots \times S_{d_r}$-equivariant isomorphism at the completion of any point in $\Delta^{ord}_{\overline{P}}(\CC)$. Let $I_\CC$ be the ideal of the $S_{d_1}\times \ldots \times S_{d_r}$-orbit corresponding to $\overline{X}$ in $\Delta^{ord}_{\overline{P}}(\CC)$, and define  the ideal $I_\CC$ similarly for $\Delta^{ord}_{\overline{P}}(\mathbb{A}^1)$.  By \cite[Lemma 2]{Knop}, completion at the ideal of a $S_{d_1}\times \ldots \times S_{d_r}$-orbit commutes with quotient by a finite group, and the result thus follows.

Now, we consider the general case. Embed $\CC\subset \mathbb{A}^n$ for some $n$.
Then we have identically to the $\mathbb{A}^1$ case that the normalization of $\Delta^{k,r}_{\overline{P}}$ is
$$\Sym^{\ell_1}\CC\times \ldots \times \Sym^{\ell_t}\CC \to \Delta^{k,r}_{\overline{P}}(\CC),$$
which under the Chow embeddings is represented by the map
$$(F_1(z),\ldots,F_t(z))\mapsto (\prod_{s=1}^tF_s^{(\overline{n_s})_1}(z),\ldots,\prod_{s=1}^tF_s^{(\overline{n_s})_r}(z))$$
where $F_i$ is a polynomial of the form $(z-\overline{v_1})\ldots(z-\overline{v_{\ell_i}})$, with $\overline{v_i}\in \CC\subset \mathbb{A}^n$, treated as an element of $$\mathbb{C}^n \times \Sym^2\mathbb{C}^n\times \ldots \times \Sym^{\ell_i}\mathbb{C}^n.$$
%A point $\overline{Y}$ in $\prod_{i=1}^{t}\Sym^{\ell_i}\CC$ corresponds to $t$ unordered collections of points $A_i=\{c_{ij}\mid 1\leq j\leq \ell_i\}$ where for each fixed $i$, the points in $A_i$ are labelled with the vector $\overline{n}_i$. 
%Hence given $c\in \CC$, we can ask for the subset of points in $\{c_{ij}\mid 1\leq i\leq t,1\leq j\leq \ell_i\}$ that are equal to $c$ and we can take their corresponding labels $\{\overline{n_i}\}_{i\in A}$ for some subset $A\subset \{1,\ldots,t\}$.
If we vary the $\overline{n_i}$ over distinct vectors in $\mathbb{C}^r$ such that $$(\underbrace{\overline{n_1},\ldots,\overline{n_1}}_{\ell_1},\ldots,\underbrace{\overline{n_t},\ldots,\overline{n_t}}_{\ell_t})\in ((\mathbb{C}^r)^k)^{\circ},$$ we obtain a family of maps from $\prod_{i=1}^t \Sym^{\ell_i}\CC$ to affine space. For fixed $\overline{Y}=(F_1,\ldots,F_t)$, the fiber-wise Jacobian matrix varies polynomially in the $\overline{n_i}$.

Let now $\overline{X}\in \Delta^{k,r}_{\overline{P}}(\CC)$ be a colored configuration and let $\overline{Y}\in \prod_{i=1}^{t}\Sym^{\ell_i}\CC$ map to $\overline{X}$ under the normalization map. Suppose the configuration $\overline{Y}$ is supported at the points $c_1,\ldots,c_u\in \CC$, such that the set of labels of points colliding at $c_i$, excluding repetition, is $\{\overline{n_j}\}_{j \in A_i}$ and $A_1 \sqcup \ldots \sqcup A_u=\{1,\ldots,t\}$. We show now that if for some fixed $i$, $\{\overline{n_j}\}_{j \in A_i}$ are linearly dependent, then the normalization map at $\overline{Y}$ is not an immersion. Indeed, as for fixed $F_1,\ldots,F_t$ the Jacobian depends polynomially on the $\overline{n_j}$, as the Jacobian drops rank whenever $\overline{n_i}\in \mathbb{N}^r$ satisfies $(\overline{n_j})_{j \in A_i}$ are linearly dependent, by \Cref{Zariskidense} this is also true when the $\overline{n_i}$ are arbitrary complex vectors.

\begin{clm}
\label{Zariskidense}
Let $\on{Mat}_{a,b}$ be the affine space of $a\times b$ matrices, where $a,b$ are positive integers. Let $\on{Mat}^k_{a,b}\subset \on{Mat}_{a,b}$ be the locus of matrices with rank at most $k$, where $k$ is a positive integer. Then the set of points $\on{Mat}^k_{a,b}\cap \on{Mat}_{a,b}(\mathbb{N})$ is Zariski-dense in $\on{Mat}^k_{a,b}$. 
\end{clm}

\begin{proof}[Proof of \Cref{Zariskidense}]
Let $\Lambda$ be the $a\times b$ matrix of rank $k$ with $k$ 1's on its diagonal and all other entries zero. Then the product $\on{Mat}_{a,a}\cdot\Lambda\cdot\on{Mat}_{b,b}$ is exactly $\on{Mat}^k_{a,b}$ and since $\on{Mat}_{a,a}(\mathbb{N})\times \on{Mat}_{b,b}(\mathbb{N})$ is Zariski-dense in $\on{Mat}_{a,a}\times \on{Mat}_{b,b}$, the claim follows.
\end{proof}
%Supposing the $\overline{n_i}$ satisfy the linear dependence condition at the configuration $\overline{c_1},\ldots,\overline{c_k}$, we will show that there is indeed a singularity as in the case $\CC=\mathbb{A}^1$. For fixed $F_1,\ldots,F_t$, the Jacobian of this map depends polynomially on the coefficients of $\overline{n_1},\ldots,\overline{n_t}$, so the set of $\overline{n_1},\ldots,\overline{n_t}$ which make the Jacobian drop rank is Zariski-closed (note that we do not need to work with matrices with an infinite number of rows, but can truncate at the stabilization point used to define $\Delta^{k,r}(\mathcal{C})$). Hence, as a subset of vectors being linearly dependent is a closed condition and we know that the rank drops when the $\overline{n_i}$ lie in $\mathbb{N}^r$ and satisfy the linear dependence condition, we have a singularity on the corresponding branch exactly like in the affine case.

To conclude, we must show that there are no other points where the Jacobian drops rank. Because $\CC \to \mathbb{A}^1$ is \'etale, the induced map
$$\Sym^{\ell_1}\CC\times \ldots \times \Sym^{\ell_r}\CC \to \Sym^{\ell_1}\mathbb{A}^1 \times \ldots \times \Sym^{\ell_r}\mathbb{A}^1$$
is \'etale and hence induces isomorphisms at the level of tangent spaces. Consider the following commutative diagram.
\begin{center}
\begin{tikzcd}
    \Sym^{\ell_1}\CC\times \ldots \times \Sym^{\ell_t}\CC\ar[r]\ar[d]&\Delta^{k,r}_{\overline{P}}(\CC)\ar[d]\\
    \Sym^{\ell_1}\mathbb{A}^1\times \ldots \times \Sym^{\ell_t}\mathbb{A}^1\ar[r]& \Delta^{k,r}_{\overline{P}}(\mathbb{A}^1)
\end{tikzcd}
\end{center}

If the bottom horizontal map has injective Jacobian at some point then the top horizontal map must have injective Jacobian at the corresponding point, and the result follows.
%be The map $\Sym^d \CC \to \Sym^d \mathbb{A}^1$ is induced by the coordinate-wise projection $$\mathbb{C}^n \times \Sym^2\mathbb{C}^n \times \ldots \times \Sym^d \mathbb{C}^n \to \mathbb{C} \times \mathbb{C} \times \ldots \mathbb{C}$$ each of which is in turn induced from the map $\mathbb{C}^n \to \mathbb{C}$ onto its first coordinate. As $\CC \to \mathbb{A}^1$ is etale, the induced map $\Sym^d \CC \to \Sym^d \mathbb{A}^1$ is etale, and hence it suffices to check whether the...
\end{proof}

\section{Isomorphisms between moduli spaces}
In this section, we apply \Cref{singthm} to show that two natural moduli spaces are not isomorphic, negatively answering a question of Farb and Wolfson \cite{Resultant}.

\begin{proof}[Proof of \Cref{notiso}]
By \cite{Resultant}, $Rat^*_{d,n}$ is isomorphic to the space of $(n+1)$-tuples of monic degree $d$ polynomials $(f_0,\ldots,f_{n})$ which share no common root. Thus $Rat^*_{d,n}$ and $Poly^{d(n+1),1}_{n+1}$ are open subsets of $\mathbb{A}^{d(n+1)}$ whose complements $(Rat^*_{d,n})^c$ and $(Poly^{d(n+1),1}_{n+1})^c$ are codimension $n$, which is at least 2 by assumption. By Hartog's Extension Theorem, an isomorphism $Rat^*_{d,n}\to Poly^{d(n+1),1}_{n+1}$ extends to a birational morphism $\mathbb{A}^{d(n+1)}\to \mathbb{A}^{d(n+1)}$, which must be an isomorphism since $\mathbb{A}^{d(n+1)}$ is integral and separated. Furthermore, the isomorphism $\mathbb{A}^{d(n+1)}\to \mathbb{A}^{d(n+1)}$ restricts to an isomorphism between their complements $(Rat^*_{d,n})^c\to (Poly^{d(n+1),1}_{n+1})^c$.

Let $\overline{P}$ be the partition of the vector $(d,\ldots,d)\in \mathbb{N}^{n+1}$ with parts
\begin{align*}\overline{P}:&(1,\ldots,1),(1,0,\ldots,0)^{d-1},(0,1,0,\ldots,0)^{d-1},\ldots,(0,0,\ldots,1)^{d-1}
\end{align*}
and let $Q$ be the partition of $(n+1)d$ with parts
\begin{align*}
Q:n+1,1^{(d-1)(n+1)}
\end{align*}
We now identify
\begin{align*}
(Rat^*_{d,n})^c&=\Delta^{(d-1)(n+1)+1,n+1}_{\overline{P}}(\mathbb{A}^1)\\
(Poly^{d(n+1),1}_{n+1})^c&=\Delta^{(d-1)(n+1)+1}_{Q}(\mathbb{A}^1).
\end{align*}
Now, as $d \ge 2$, from \Cref{singthm} we immediately conclude that $\Delta^{(d-1)(n+1)+1}_{Q}(\mathbb{A}^1)$ is singular in codimension $1$ when a point of multiplicity $1$ and $n+1$ collide, whereas $\Delta^{(d-1)(n+1)+1,n+1}_{\overline{P}}(\mathbb{A}^1)$ is first singular in codimension $n$, along the strata $\Delta^{(d-2)(n+1)+2,n+1}_{\overline{R}}(\mathbb{A}^1)$ where $\overline{R}$ is the vector partition
\begin{align*}
    \overline{R}:(1,\ldots,1)^2,(1,0,\ldots,0)^{d-2},(0,1,0,\ldots,0)^{d-2},\ldots,(0,0,\ldots,1)^{d-2}.
\end{align*}
\end{proof}

\section{Equivalence of Constructions and Deligne Categories}
\label{Dcategories}
In this section, we show that the constructions from \Cref{3Con} are equivalent, and in particular prove \Cref{comparison}.

\subsection{Equivalence of Construction 1 and Construction 2}
In this subsection we show that Constructions 1 and 2 from \Cref{3Con} are equivalent.

First we show the equivalence for $\mathbb{A}^n$, i.e. when $A=\mathbb{C}[x_1,\ldots,x_n]$.  Construction 2 at the level of rings yields the subalgebra of $\Gamma((\mathbb{C}^k)^{\circ})\otimes\mathbb{C}[\{x_{i,j}\}_{1 \le i \le k,1 \le j \le n}]$ generated by expressions of the form $$\sum_{i=1}^km_i\prod_{j=1}^n x_{i,j}^{s_j}$$ where $s_1,\ldots,s_n \in \mathbb{Z}_{\ge 0}$.

Let $v_1,\ldots,v_n$ be the standard basis vectors of $\mathbb{C}^n$, and write $\overline{x_i}=\sum x_{i,j}v_j$ for the vector in $\mathbb{C}^n$ with entries $x_{i,j}$ with $1 \le j \le n$. By \Cref{newtlem}, at the level of rings, Construction 1 yields the subalgebra of $\Gamma((\mathbb{C}^k)^{\circ})\otimes\mathbb{C}[\{x_{i,j}\}_{1 \le i \le k,1 \le j \le n}]$ generated by the coefficients in the $\prod v_i^{\mu_i}$-basis of expressions of the form
$$m_1\overline{x_1}^\ell+\ldots+m_k\overline{x_k}^\ell.$$ The subalgebra from Construction 1 is clearly contained in the subalgebra from Construction 2, and conversely the expression $\sum_{i=1}^km_i\prod_{j=1}^n x_{i,j}^{s_j}$ appears as a multiple of the $v_1^{s_1}\ldots v_n^{s_n}$-coefficient of $m_1\overline{x_1}^\ell+\ldots+m_k\overline{x_k}^\ell$ with $\ell=s_1+\ldots+s_n$, so the two constructions in fact yield the same subalgebra.

Finally, suppose that $\Spec A\subset \mathbb{A}^n$ is a variety with $A=\mathbb{C}[x_1,\ldots,x_n]/I$. Then Construciton 1 yields the image of the composite
$$(\mathbb{C}^k)^{\circ}\times(\Spec A)^k \hookrightarrow (\mathbb{C}^k)^{\circ}\times(\mathbb{A}^n)^k\to (\mathbb{C}^k)^{\circ}\times\prod_{i=1}^{N_k}\Sym^i\mathbb{C}^n.$$
which by \Cref{newtlem} corresponds to the subalgebra of $\Gamma((\mathbb{C}^k)^{\circ})\otimes A^{\otimes k}$ generated by the coefficients of $m_1\overline{x_1}^\ell+\ldots +m_k\overline{x_k}^\ell$. The same argument as above shows that this is the subalgebra generated by all expressions of the form $\sum_{i=1}^k m_i \prod_{j=1}^n x_{i,j}^{s_j}$, and as every element of $A$ is a $\mathbb{C}$-linear combination of monomials $x_1^{s_1}\ldots x_n^{s_n}$, this subalgebra is equivalently described as the subalgebra generated by elements $$\sum m_i(1^{\otimes i-1}\otimes a \otimes 1^{\otimes k-i}),$$
which is precisely Construction 2.

\subsection{Equivalence of Construction 2 and Construction 3}
In this subsection we prove \Cref{comparison}.

To do this, we need to unwind the construction in the Deligne category. We fix complex $m_1,\ldots,m_k\in\mathbb{C}\backslash \mb{Z}_{\ge 0}$ summing to $d\in\mathbb{C}\setminus \mathbb{Z}_{\ge 0}$ and let $I=I(m_1,\ldots,m_k)$. 

The Deligne category $\on{Rep}(S_d)$ is semi-simple with irreducible objects $X_\lambda$ parametrized by young tableaux $\lambda$. We write $\1=X_{\emptyset}$,  $\mathfrak{h}'$ for the $X_\lambda$ corresponding to the standard representation, and let $\mathfrak{h}=\1\oplus \mathfrak{h}'$ be the object corresponding to the permutation representation. The following fact will allow us to work with elements in algebras.

\begin{clm}
\label{equivclm}
The subcategory of $\on{Rep}(S_d)$ generated by $\1$ is equivalent to the category $\on{Vect}_k$ of finite-dimensional vector spaces given by the functor $\on{Hom}(\1, -)$. 
\end{clm}
\begin{proof}
The inverse to $\on{Hom}(\1, -)$ takes a vector space $V$ to $V \otimes \1$, and the result follows since $\on{Hom}(\1,\1)=\mathbb{C}$.
\end{proof}

Therefore, if we have an algebra inside $\on{Ind}(\on{Rep}(S_d))$ which as an object is a direct sum of $\1$'s, then we may treat it as an honest $\mathbb{C}$-algebra by applying $\on{Hom}(\1, -)$. 
\begin{defn}
Given an object $B$ in $\on{Rep}(S_d)$ or $\on{Ind}(\on{Rep}(S_d))$, let $B^{S_d}$ denote the $\1$-isotypic component of $B$.
\end{defn}

%Now, given any object $B$ in $\on{Rep}(S_m)$ or $\on{Ind}(\on{Rep}(S_m))$, we have Schur functor $B \mapsto B^{S_m}$, which maps $B$ to the $\1$-isotypic component of $B$. 
Using the fact that $\on{Rep}(S_d)$ is semisimple, we see the functor $(-)^{S_d}$ has the property that if $B$ is an algebra in $\on{Rep}(S_d)$ or $\on{Ind}(\on{Rep}(S_d))$ and if $J$ is an ideal in $B$, then $J^{S_d}$ is an ideal in $B^{S_d}$, and $(B/J)^{S_d} \cong B^{S_d}/J^{S_d}$. Finally, we have $\on{Hom}(\1,B)\cong \on{Hom}(\1,B^{S_d})$ for any object $B$, and by \Cref{equivclm} we have $$\on{Hom}(\1,B/J)=\on{Hom}(\1,B)/\on{Hom}(\1,J)=\on{Hom}(\1,B^{S_d})/\on{Hom}(\1,J^{S_d}).$$
 \begin{proof}[Proof of \Cref{comparison}]
 
 We are interested in computing the algebra
 $$\on{Hom}(\1,A^{\otimes d}/I)=\on{Hom}(\1,(A^{\otimes d})^{S_d})/\on{Hom}(\1,I^{S_d}).$$
 We first characterize $I^{S_d}$.
 
Consider the diagram
\begin{center}
\begin{tikzcd}
A^{\otimes d}\ar[r, mapsto, "\on{Res}"]& A^{\otimes m_1}\boxtimes \ldots \boxtimes A^{\otimes m_k}& (A^{\otimes m_1})^{S_{m_1}}\boxtimes \ldots \boxtimes (A^{\otimes m_k})^{S_{m_k}}\ar[l]\\
(A^{\otimes d})^{S_d}\ar[u]\ar[r, mapsto, "\on{Res}"] & \on{Res}((A^{\otimes d})^{S_d})\ar[u]\ar[ur]&
\end{tikzcd}
\end{center}

Given a subobject $J'$ of a commuatative algebra $B$ in $\on{Rep}(S_d)$ or $\on{Ind}(\on{Rep}(S_d))$, the ideal $J''$ generated by $J'$ is the image of the map $B\otimes J'\to B\otimes B\to B$. 
\begin{clm}
\label{Resclm}
$I^{S_d}$ is the largest ideal $J'\subset (A^{\otimes d})^{S_d}$ such that $\on{Res}(J')\subset (J_{m_1})^{S_{m_1}}\boxtimes\ldots \boxtimes (J_{m_1})^{S_{m_k}}$.
%If $J'\subset (A^{\otimes m})^{S_m}$ is an ideal such that $\on{Res}(J')$ lies in $J_{m_1}\boxtimes\ldots\boxtimes J_{m_k}$, then the $A^{\otimes m}$-ideal $J''$ generated by $J'$ lies in $I$. 
\end{clm}
\begin{proof}
We clearly have $\on{Res}(I^{S_d})\subset \on{Res}(I)\subset J_{m_1}\boxtimes \ldots \boxtimes J_{m_k}$, so taking invariants we see that $\on{Res}(I^{S_d})\subset(J_{m_1})^{S_{m_1}}\boxtimes\ldots \boxtimes (J_{m_1})^{S_{m_k}}$. We now show that $I^{S_d}$ is the largest ideal in $(A^{\otimes d})^{S_d}$ with this property.

Let $J'\subset (A^{\otimes d})^{S_d}$ be an ideal such that $\on{Res}(J')$ lies in $J_{m_1}\boxtimes\ldots\boxtimes J_{m_k}$, and let $J''$ be the $A^{\otimes d}$-ideal generated by $J'$. Since $\on{Res}$ is an exact fuctor it commutes with taking images, so we have $\on{Res}(J'')$ is the ideal generated by $\on{Res}(J')$. As $\on{Res}(J')\subset J_{m_1} \boxtimes \ldots\boxtimes J_{m_k}$, we deduce $\on{Res}(J'')\subset J_{m_1} \boxtimes \ldots \boxtimes J_{m_k}$, so $J'' \subset I$ by definition of $I$. In particular, this implies $J' \subset I^{S_d}$.
\end{proof}

Now, $\on{Res}$ is an equivalence of categories when restricted to the $1$-isotypic part of $\on{Rep}(S_d)$ and $\on{Rep}(S_{m_1})\boxtimes \ldots \boxtimes \on{Rep}(S_{m_k})$, so $$\on{Hom}(\1,(A^{\otimes d})^{S_d})=\on{Hom}(\1,\on{Res}((A^{\otimes d})^{S_d}))$$
and therefore the ideal $\on{Hom}(\1,I^{S_d})\subset \on{Hom}(\1,(A^{\otimes d})^{S_d})$ is the kernel of the  product of the multiplication maps
\begin{align*}
    \on{Hom}(\1,A^{\otimes d})&=\on{Hom}(\1,\on{Res}(A^{\otimes d}))\\&\to \on{Hom}(\1,(A^{\otimes m_1})^{S_{m_1}})\otimes \ldots \otimes \on{Hom}(\1,(A^{\otimes m_k})^{S_{m_k}})\to A\otimes\cdots\otimes A.
\end{align*}
%Therefore, $I^{S_m}$ is the ideal of $(A^{\otimes m})^{S_m}$ corresponding to the kernel of the map of $\mathbb{C}$-algebras $$(A^{\otimes m})^{S_m}=\on{Res}((A^{\otimes m})^{S_m})\to A\otimes \ldots \otimes A.$$

By \cite[Proposition 4.1]{E14}, we have
$$\on{Hom}(\1,A^{\otimes d})\cong S(A)/(1_A=d),$$
where $S(A)$ is the symmetric algebra generated by $A$. The relation $(1_A=d)$ indicates the unit $1_A\in A$ is identified with $d\in A^{0}\cong \mb{C}$. We first explicitly describe the map $$S(A)/(1_A=d) \to S(A)/(1_A=m_1)\otimes \ldots \otimes S(A)/(1_A=m_k).$$

To do this, we note that $A^{\otimes d}$ is canonically the quotient of the tensor algebra $T(A \otimes \mathfrak{h})$ by certain relations \cite[Proposition 4.3]{E14}. We claim that $(A^{\otimes d})^{S_d}$ is generated by the invariant generators arising from the composite
$$\overline{a}:\1 \xrightarrow{a \otimes \on{Id}} A \otimes \1 \to A \otimes \mathfrak{h} \to T(A\otimes \mathfrak{h})\to A^{\otimes d}.$$

Indeed, we know that $\on{Hom}(\1,(A^{\otimes d}))\cong S(A)/(1_A=d)$, and under this isomorphism the elements of $A\subset S(A)/(1_A=d)$ correspond to precisely to the homorphisms described above.

Now, we have the diagram
\begin{center}
\begin{tikzcd}
A^{\otimes d} \ar[r,mapsto,"Res"]& A^{\otimes m_1}\boxtimes \ldots \boxtimes A^{\otimes m_k}\\
A\otimes \mathfrak{h}\ar[u]\ar[r,mapsto,"Res"]&A\otimes (\mathfrak{h}_1\oplus \ldots \oplus \mathfrak{h}_k)\ar[u]\\
A\otimes \1 \ar[r,mapsto,"Res"]\ar[u]& A\otimes (\1\oplus \ldots \oplus \1)\ar[u]
\end{tikzcd}
\end{center}
where by $\mathfrak{h}_i$ we really mean $\1\boxtimes \ldots \boxtimes \1 \boxtimes \mathfrak{h}_i\boxtimes \1 \boxtimes \ldots \boxtimes \1$ where $\mathfrak{h}_i$ is in the $i$'th place. Thus the map
$$S(A)/(1_A=d)\to S(A)/(1_A=m_1)\otimes \ldots \otimes S(A)/(1_A=m_k)$$
is given on generators by
$$a\mapsto (a\otimes 1 \otimes \ldots \otimes 1)+(1 \otimes a \ldots \otimes 1)+\ldots+(1 \otimes 1 \otimes \ldots \otimes a).$$
Now, we consider the multiplication map
$A^{\otimes m_i}\to A\otimes \1$. By definition, this is given by factoring the map $T(A\otimes \mathfrak{h})\to A\otimes\1$ through $A^{\otimes m_i}$, where the map is given by setting for each $j$
$$(A\otimes \mathfrak{h})^{\otimes j}\to A\otimes \1$$
induced by the maps $A^{\otimes j}\to A$ and $\mathfrak{h}^{\otimes j}\to \1$ given by the basis element in $\mb{C}P_{j,0}$ \cite[Definition 2.11]{CO11} corresponding to the finest partition of $j$.

The multiplication maps induce algebra maps $S(A)/(1_A=m_i)\to A$ which we want to determine at the level of generators. Given a generator $a \in A\subset S(A)$, we compute the map to $A$ as arising from the composite
$\1 \to A\otimes \1 \to A \otimes \mathfrak{h}\to A \otimes \1$
where the composite map is $a$ on the first coordinate, and $\1 \to \mathfrak{h} \to \1$ on the second coordinate, which is multiplication by $m_i$ \cite[Definition 2.11]{CO11}. Therefore, the map of algebras, $S(A)/(1_A=m_i)\to A$ takes each generator $a \in A\subset S(A)$ to $m_ia \in A$.

Therefore, the composite map $$S(A)/(1_A=d)\to S(A)/(1_A=m_1)\otimes \ldots \otimes S(A)/(1_A=m_k) \to \underbrace{A \otimes \ldots \otimes A}_k$$
takes $$a \mapsto m_1(a\otimes 1 \otimes \ldots \otimes 1)+m_2(1 \otimes a \otimes \ldots \otimes 1)+\ldots +m_k(1 \otimes 1 \otimes \ldots \otimes a),$$
and the quotient $\on{Hom}(\1,(A^{\otimes d})^{S_d})/\on{Hom}(\1,I^{S_d})$ is therefore isomorphic to the subalgebra of $A^{\otimes k}$ generated by such elements, which is precisely Construction 2.
\end{proof}
\bibliography{refs}
\bibliographystyle{alpha}

\end{document}